\documentclass{article}[12pt]
\usepackage{url, graphicx, xcolor, hyperref, geometry}
\usepackage{latexsym,amsbsy}
\usepackage{lipsum}
\usepackage{amsfonts}
\usepackage{graphicx}
\usepackage{epstopdf}
\usepackage{algorithmic}
\usepackage{tikz}
\usetikzlibrary{matrix}
\usetikzlibrary{shapes,patterns,arrows,snakes,decorations.shapes}
\usetikzlibrary{positioning,fit,calc}
\usetikzlibrary{plotmarks}
\usepackage{amssymb,amsmath,amsthm}
 \usepackage{float}
 \usepackage{todonotes}

\newtheorem{theorem}{Theorem}[section]

\newtheorem{remark}{Remark}[section]

\newtheorem{definition}{Definition}[section]
\newtheorem{proposition}{Proposition}[section]

\hypersetup{
    colorlinks,
    linkcolor={blue},
    citecolor={blue},
    urlcolor={blue}
}

\begin{document}

\title{A Vanka-based  parameter-robust multigrid relaxation for the Stokes-Darcy Brinkman problems}

\author{Yunhui He\thanks{Department of Computer Science, The University of British Columbia, Vancouver, BC, V6T 1Z4, Canada,   \tt{yunhui.he@ubc.ca}.}  }

\maketitle

 \begin{abstract}
We propose  a block-structured multigrid relaxation scheme for solving the Stokes-Darcy Brinkman equations discretized by the marker and cell scheme.   An element-based additive Vanka smoother is used to solve the corresponding  shifted Laplacian operator. Using local Fourier analysis, we present the stencil for the additive Vanka smoother and derive an optimal smoothing factor for Vanka-based Braess-Sarazin relaxation  for the Stokes-Darcy Brinkman equations.  Although the optimal damping parameter is dependent on meshsize and physical parameter,  it is very close to one.  Numerical results   of two-grid and V(1,1)-cycle are presented, which show  high efficiency of the proposed relaxation scheme and its robustness  to physical parameters and the meshsize.  Using a damping parameter equal to one gives almost the same results as these for the optimal damping parameter at a lower computational overhead.  
\end{abstract}

\section{Introduction}
 
The numerical solution of fluid flow problems  is an important  topic in computational science and engineering, which has received much attention in the last few decades \cite{bathe2007benchmark,brandt1979multigrid,brandt2011multigrid,connor2013finite,vanka1986block,wesseling2001geometric}.   Stokes-Darcy Brinkman problem is one of them used to model fluid motion in porous media with fractures. The discretization of the fluid flow problems often leads to a saddle-point system, which is ill-conditioned.  Designing fast numerical solution of these problems is often challenging due to the small magnitude of physical parameters of the model.

We consider the multigrid numerical solution of the  Stokes-Darcy Brinkman equations
\begin{subequations}
\label{eq:SDB}
\begin{align}
  - \epsilon^2 \Delta\boldsymbol{u} +\boldsymbol{u} +  \nabla  p =&\boldsymbol{f} \qquad \text{in}\,\, \Omega \label{eq:stokes-DB-one}\\
  \nabla  \cdot \boldsymbol{u}=&g \qquad \text{in}\,\,  \Omega \label{eq:stokes-DB-two} \\
   \boldsymbol{u} =&0  \qquad \text{on}\,\,  \partial \Omega \label{eq:stokes-DB-three},  
\end{align}
\end{subequations}
where $\epsilon \in(0,1]$.    The source term $g$  is assumed to satisfy the solvability condition
\begin{equation*}
 \int_{\Omega} g\, d\Omega =0.
\end{equation*} 
Then, equations \eqref{eq:stokes-DB-one}, \eqref{eq:stokes-DB-two}, and \eqref{eq:stokes-DB-three} have a unique solution.  

A variety of discretization schemes are available for equations \eqref{eq:stokes-DB-one}, \eqref{eq:stokes-DB-two}, and \eqref{eq:stokes-DB-three}, including finite element methods \cite{gulbransen2010multiscale,vassilevski2014mixed,xie2008uniformly,xu2010new,zhai2016new,zhang2009low,zhao2020new}, finite difference techniques \cite{he2017finite,sun2019stability}, and divergence-conforming B-spline methods \cite{evans2013isogeometric}.  When $\epsilon= 0$, the model problem is reduced to the Darcy problem \cite{arraras2021multigrid,harder2013family}.  For $\epsilon \in(0,1]$, designing a robust discretization and numerical solver is  challenging.  The convergence rate deteriorates as the Stokes-Darcy Brinkman becomes Darcy-dominating when certain stable Stokes elements are used \cite{hannukainen2011computations}, for example, Taylor–Hood element.  While, as the Stokes-Darcy Brinkman problem becomes Stokes-dominating when Darcy stable elements such as the lowest order Raviart-Thomas elements are used,  degradation on convergence is observed \cite{mardal2002robust}. 

Upon discretization, large-scale indefinite linear systems typically need to be solved, at times repeatedly.  For saddle-point systems,
within the context of multigrid, there are several effective block-structured relaxation  schemes for solving such linear systems, such as Braess-Sarazin  smoother \cite{braess1997efficient, he2018local,MR1810326},   distributive smoother \cite{chen2015multigrid,he2018local},  Schwarz-type smoothers \cite{schoberl2003schwarz}, Vanka smoother \cite{claus2021nonoverlapping,MR3217219,manservisi2006numerical,molenaar1991two,wobker2009numerical},  and Uzawa-type relaxation \cite{MR1451114,MR3217219,MR1302679,luo2017uzawa}.

We note also that a number of effective preconditioning methods are available for the Stokes-Darcy Brinkmans problems, for example the scalable block diagonal preconditioner \cite{vassilevski2013block}, and Uzawa-type preconditioning \cite{kobelkov2000effective,sarin1998efficient}.  
Multigrid methods are studied in depth \cite{butt2018multigrid,coley2018geometric,kanschat2017geometric,larin2008comparative,olshanskii2012multigrid}.   Braess-Sarazin, Uzawa, and Vanka smoothers within multigrid with finite element discretization have been discussed \cite{larin2008comparative}. However, the convergence rate is highly dependent on physical parameters.  A Gauss–Seidel smoother based on a Uzawa-type iteration is studied \cite{MR3217219}, where the authors provide an upper bound on the smoothing factor.   Moreover,   the performance of Uzawa with  a Gauss–Seidel type coupled Vanka smoother \cite{vanka1986block} has been investigated  \cite{MR3217219}, in which the  pressure and the velocities in a grid cell, are updated simultaneously, showing that the actual convergence of the W-cycle of Uzawa is approximately the same as that obtained by the Vanka smoother.

Our interest is in  the marker and cell scheme (MAC), a finite difference method on a staggered mesh. On a uniform mesh discretization, the method is   second-order accurate for both velocity and pressure \cite{sun2019stability}. We propose a Vanka-type  Braess-Sarazin relaxation (V-BSR) scheme for the Stokes-Darcy Brinkman equations discretized by the MAC scheme on  staggered meshes.  In contrast to the Vanka smoother \cite{MR3217219}, our work builds an algorithm that decouples velocity and pressure, which is often preferred considering the cost efficiency.  Specifically, in our relaxation scheme, the shifted Laplacian  operator, $ - \epsilon^2 \Delta\boldsymbol{u} +\boldsymbol{u}$,  is solved by an additive Vanka-type smoother.  Instead of solving many subproblems involved in Vanka setting, we derive the stencil of the Vanka smoother, which means that we can form the global matrix of the Vanka smoother. As a result,  in our multigrid method we only have matrix-vector products.  This represents significant savings compared to traditional methods that require computationally expensive exact solves; in V-BSR, we solve the Schur complement system by only two or three iterations of the Jacobi method, which achieves the same performance as that of exact solve.  We apply local Fourier analysis (LFA) to select the multigrid damping parameter and predict the actual multigrid performance.  From this analysis, we derive an optimal damping parameter and optimal smoothing factors. Those parameters are dependent on physical parameters and the meshsize, which means that we can propose adaptive damping parameter in each multigrid level. The optimal parameter turns out to be close to one and relatively insensitive to physical parameters and meshsize. This allows for an easy choice of an approximately optimal damping parameter. We  quantitatively compare the results with optimal parameter and the value of one from LFA and present numerical results of two-grid and V-cycle multigrid to validate the high efficiency of our methods. Our V-cycle results outperform these of Uzawa and Vanka smoothers \cite{MR3217219}, especially for small $\epsilon$.

The rest of the work is organized as follows.  In Section \ref{sec:Discretization-relaxation} we review the MAC scheme for our model problem and propose the afore-mentioned Vanka-based Braess-Sarazin relaxation.  We apply LFA to study the smoothing process in Section \ref{sec:LFA}, where optimal LFA smoothing factor is derived.  In Section \ref{sec:numerical} we present our LFA predictions for the two-grid method and actual multigrid performance.  Finally, we draw conclusions in Section \ref{sec:con}.

\section{Discretization and relaxation}\label{sec:Discretization-relaxation}

As mentioned in the Introduction, we  use  throughout the well-known MAC  scheme to solve~\eqref{eq:SDB}. For  the discretization  of \eqref{eq:SDB},   a staggered mesh is needed to guarantee numerical stability.  The discrete unknowns $u,v,p,$ are placed in different locations; see Figure \ref{fig:MAC-stokes-brinkman}.  The stability and convergence of the MAC scheme for this problem has been studied \cite{sun2019stability}.

\begin{figure}[H]
\centering
\includegraphics[width=0.5\textwidth]{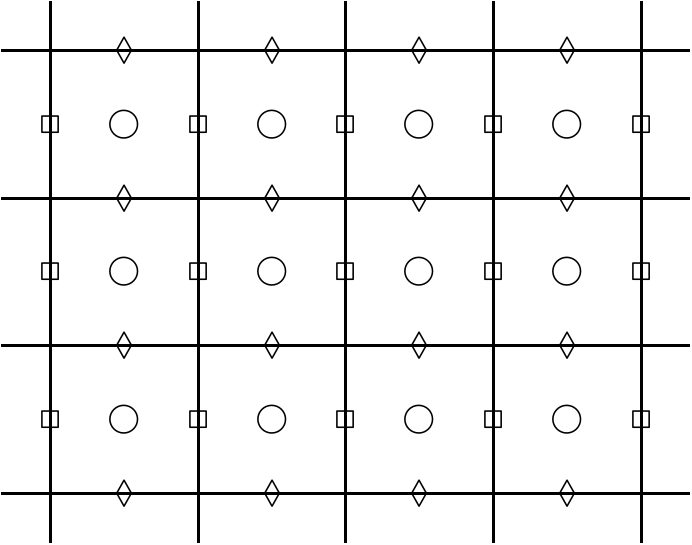}
 \caption{The  location of the unknowns in the staggered grid: $\Box-u,\,\, \lozenge-v,  \,\, \bigcirc-p$.}\label{fig:MAC-stokes-brinkman}
\end{figure}

The stencil representation of MAC for the Stokes-Darcy Brinkmann  equations is
 \begin{equation}\label{eq:Kh--stencil-operator}
   \mathcal{K}_h   =\begin{pmatrix}
      - \epsilon^2\Delta_{h}  + I                         & 0                                          &   (\partial_{x})_{h/2}\\
      0                                          & - \epsilon^2 \Delta_{h}   + I              &    (\partial_{y})_{h/2} \\
      - (\partial_{x})_{h/2}        &   - (\partial_{y})_{h/2}             & 0
    \end{pmatrix},
  \end{equation}
where  
\begin{equation*}
   -\Delta_{h} =\frac{1}{h^2}\begin{bmatrix}
       & -1 &  \\
      -1 & 4 & -1 \\
        & -1 &
    \end{bmatrix},\quad
     (\partial_{x})_{h/2} =\frac{1}{h}\begin{bmatrix}
       -1& 0 & 1 \\
    \end{bmatrix},\quad
     (\partial_{y})_{h/2} =\frac{1}{h}\begin{bmatrix}
       1 \\
      0 \\
      -1
    \end{bmatrix}.
\end{equation*}

After discretization, the corresponding linear system is 
\begin{equation}\label{eq:linear-system}
\mathcal{K}_h \boldsymbol{x}_h=\begin{pmatrix}
\mathcal{A} & \mathcal{B}^T\\
\mathcal{B} & 0
\end{pmatrix}
\begin{pmatrix}
\boldsymbol{u}_h
\\
p_h
\end{pmatrix}
=\begin{pmatrix}
\boldsymbol{f}_h
\\
g_h
\end{pmatrix}=b_h,
\end{equation} 
where $\mathcal{A}$ is the matrix corresponding to the discretization of $ - \epsilon^2 \Delta\boldsymbol{u} +\boldsymbol{u}$, and $\mathcal{B}^T$ is the discrete gradient.

In order to solve \eqref{eq:linear-system} efficiently by multigrid we use BSR, with the smoother
 \begin{equation}\label{eq:Mh-form}
\mathcal{M}_h= \begin{pmatrix}
\mathcal{C} & \mathcal{B}^T\\
\mathcal{B} & 0
\end{pmatrix},
\end{equation}
where $\mathcal{C}$ is an approximation to $\mathcal{A}$. In the context of preconditioning, such an approach is known as {\em constraint preconditioning} \cite{keller2000constraint,chidyagwai2016constraint,rees2007preconditioner}  and it has received quite a bit of attention due to its attactive property of computing interim approximate solutions that satisfy the constraints.

A number of studies  \cite{he2018local,YH2021massStokes} have shown that the efficiency of solving the Laplacian will determine the convergence of BSR.  To construct an efficient approximation $\mathcal{C}$, we first investigate the discrete operator $-\epsilon^2 \Delta u +u $ denoted by 
 \begin{equation}\label{eq:positive-shift-Laplace-system}
L=  A+ I,
\end{equation} 
where  $A$ corresponds to the five-point discretization of  operator  $-\epsilon^2 \Delta u$. The stencil notation for the discrete operator $-\epsilon^2 \Delta u +u $ is 

\begin{equation}\label{eq:shift-Laplace-stencil}
L= \frac{\epsilon^2}{h^2}
\begin{bmatrix}
& -1 & \\
-1 & 4+\frac{h^2}{\epsilon^2} & -1 \\
& -1 & 
\end{bmatrix}=\frac{\epsilon^2}{h^2}
\begin{bmatrix}
& -1 & \\
-1 & 4+r & -1 \\
& -1 & 
\end{bmatrix},
\end{equation}
where $r=\frac{h^2}{\epsilon^2}$. When $r=0$,  \eqref{eq:shift-Laplace-stencil} reduces to the discretization of  $-\epsilon^2 \Delta u$.
 
Recently, we proposed an additive element-wise Vanka smoother \cite{CH2021addVanka} for $\Delta u$. Our current goal is to extend our approach to \eqref{eq:positive-shift-Laplace-system}. An immediate challenge here contrary to \cite{CH2021addVanka} is the difference in scale between the discretized scaled Laplacian and the identity operator.

Denote the element-wise smoother as $M_e$, which has the form
 \begin{equation}\label{eq:Vanka-operator}
M_e = \sum_{j=1}^{N} V_j^T D_j L_j^{-1} V_j,
\end{equation} 
where $D_j=\frac{1}{4}I$ with $I$ be the $4\times 4$-identity matrix,  $L_j$ is the coefficient matrix of $j$-th subproblem defined for one element, and $V_j$ is a restriction operator mapping the global vector to the $j$-th subproblem.
We consider
\begin{equation*}
\mathcal{C}^{-1} = 
\begin{pmatrix}
M_e & 0\\
0     & M_e
\end{pmatrix}.
\end{equation*}
The relaxation scheme  for \eqref{eq:linear-system}  is 
\begin{equation}\label{eq:BSR-relax-scheme}
\boldsymbol{x}^{k+1}_h = \boldsymbol{x}^k_h+\omega \mathcal{M}^{-1}_h(b_h- \mathcal{K}_h\boldsymbol{x}^k_h).
\end{equation} 
We refer to the above relaxation as {\em Vanka-based Braess-Sarazin relaxation} (V-BSR). 

Let $b_h- \mathcal{K}_h\boldsymbol{x}^k_h=( r_{ \boldsymbol{u}}, r_p)$.  In \eqref{eq:BSR-relax-scheme}, we need to solve for
 $(\delta \boldsymbol{u}, \delta p)=\mathcal{M}^{-1}_h(r_{\boldsymbol{u}}, r_p)$ by
\begin{eqnarray}
  (\mathcal{B}\mathcal{C}^{-1}\mathcal{B}^{T})\delta p&=&\mathcal{B}\mathcal{C}^{-1} r_{ \boldsymbol{u}}-  r_{p}, \label{eq:solution-schur-complement}\\
  \delta \boldsymbol{u}&=&  \mathcal{C}^{-1}(r_{ \boldsymbol{u}}-\mathcal{B}^{T}\delta p).\nonumber
\end{eqnarray}

Solving \eqref{eq:solution-schur-complement} exactly is prohibitive and impractical in the current context, and it has been shown in a few studies \cite{he2018local,MR1810326}  that an inexact solve  can be applied and perform well. In the sequel we will present a smoothing analysis for the exact solve, but in practice, for assessing the performance of the multigrid scheme we apply a few iterations of weighted Jacobi to \eqref{eq:solution-schur-complement}. 

The relaxation error operator for \eqref{eq:BSR-relax-scheme} is given by
\begin{equation}\label{eq:relaxation-error-operator}
\mathcal{S}_h = I-\omega \mathcal{M}^{-1}_h\mathcal{K}_h,
\end{equation}
where $\omega$ is a damping parameter to be determined. 

For a two-grid method, the error propagation operator  is 
\begin{equation}\label{eq:two-grid-error-operator}
E_h=S^{\nu_2}_h (I- P_h(L_{2h})^{-1} R_h L_h) S^{\nu_1}_h,
\end{equation}
 where $L_{2h}$ is rediscretization for the coarse-grid operator and the integers $\nu_1$ and $\nu_2$ are the numbers of pre-  and post-smoothing steps.   For simplicity, we denote the overall number of those steps by $\nu=\nu_1+\nu_2$.  We consider  simple restriction operators using  six points  for the $u$ and $v$ components of the velocity,  that is, 
 \begin{equation*} 
    R_{h,u}   =    \frac{1}{8}\begin{bmatrix}
   1 &       & 1  \\
   2 &   \star   & 2   \\
  1 &      & 1  
   \end{bmatrix},
 \quad 
   R_{h,v}   =    \frac{1}{8}\begin{bmatrix}
   1 &    2  & 1  \\
    &   \star   &    \\
  1 &   2   & 1  
   \end{bmatrix},
\end{equation*}
where the $\star$ denotes the position (on the coarse grid) at which the discrete operator is applied.  For interpolation, we take  $P_{h,u} =4 R^T_{h,u}$ and     $P_{h,v} =4 R^T_{h,v}$.
For the  restriction for the pressure, we use 
 \begin{equation*} 
   R_{h,p}   =    \frac{1}{4}\begin{bmatrix}
   1 &      & 1  \\
    &   \star   &    \\
  1 &    & 1  
   \end{bmatrix},
\end{equation*}
and $P_{h,p}=4R^T_{h,p}$.  Consequently,
\begin{equation*}
R_h =  \begin{pmatrix}
   R_{h,u} & 0     & 0  \\
  0  &   R_{h,v}  & 0   \\
  0 &  0  & R_{h,p} 
   \end{pmatrix}, \quad P_h = 4 R^T_h. 
\end{equation*}  

\section{Local Fourier analysis}\label{sec:LFA}
 
 To identify a proper parameter $\omega$  in \eqref{eq:BSR-relax-scheme} to construct fast multigrid methods, we apply LFA \cite{wienands2004practical,trottenberg2000multigrid}  to examine the multigrid relaxation scheme. 
The LFA smooting factor, see Definition \ref{def:LFA-mu},  often offers a sharp prediction of actual multigrid performance.  
 
\begin{definition}
Let $L_h =[s_{\boldsymbol{\kappa}}]_{h}$  be a scalar stencil operator acting on grid ${G}_{h}$ as
\begin{equation*}
  L_{h}w_{h}(\boldsymbol{x})=\sum_{\boldsymbol{\kappa}\in{V}}s_{\boldsymbol{\kappa}}w_{h}(\boldsymbol{x}+\boldsymbol{\kappa}h),
\end{equation*}
where  $s_{\boldsymbol{\kappa}}\in \mathbb{R}$ is constant,   $w_{h}(\boldsymbol{x}) \in l^{2} ({G}_{h})$, and  ${V}$ is a finite index set. 
Then, the  symbol of $L_{h}$ is defined as:
\begin{equation}\label{eq:symbol-calculation-form}
 \widetilde{L}_{h}(\boldsymbol{\theta})=\displaystyle\sum_{\boldsymbol{\kappa}\in{V}}s_{\boldsymbol{\kappa}}e^{i \boldsymbol{\theta}\cdot\boldsymbol{\kappa}},\qquad i^2=-1. 
\end{equation} 
\end{definition}

We consider standard coarsening. The low and high frequencies are given by  $$\boldsymbol{\theta} \in  T^{\rm{L}} =\left[-\frac{\pi}{2}, \frac{\pi}{2}\right)^d;  \qquad \boldsymbol{\theta} \in  T^{\rm{H}} =\left[-\frac{\pi}{2}, \frac{3\pi}{2}\right)^d \setminus T^{\rm{L}}.$$

\begin{definition}\label{def:LFA-mu}
We define the LFA smoothing factor for relaxation error operator $\mathcal{S}_h$ as

\begin{equation*}
\mu_{\rm loc}(\mathcal{S}_h) = \max_{\boldsymbol{\theta} \in T^{\rm{H}}}\{\rho(\widetilde{\mathcal{S}}_h(\boldsymbol{\theta}))\},
\end{equation*}
where $\rho(\widetilde{S}_h(\boldsymbol{\theta}))$ stands for the spectral radius of $\widetilde{\mathcal{S}}_h(\boldsymbol{\theta})$.
\end{definition}

The symbol of  $\mathcal{S}_h$ defined in \eqref{eq:relaxation-error-operator} is a $3\times 3$ matrix since  $\mathcal{K}_h$ is a $3\times 3$ block system; see \eqref{eq:Kh--stencil-operator}. The same holds for  $\mathcal{M}_h$, see \eqref{eq:Mh-form}, and the symbol of each block is a scalar.  For more details on how to compute the symbol of 
$\mathcal{S}_h$, refer to other studies \cite{farrell2021local,he2018local}.  Since  $\mu_{\rm loc}(\mathcal{S}_h)$ is a function of the parameter $\omega$, we are interested in minimizing $\mu_{\rm loc}(\mathcal{S}_h)$ over $\omega$ to obtain a fast convergence speed. We define the optimal smoothing factor as 
 \begin{equation*}
\mu_{\rm opt} =\min_{\omega}\mu_{\rm loc}(\mathcal{S}_h).
\end{equation*}

For the two-grid error operator $E_h$ defined in \eqref{eq:two-grid-error-operator},   the two-grid LFA convergence factor is 
\begin{equation}\label{eq:LFA-two-grid-convergence-factor}
\rho_h(\nu)= \max_{\boldsymbol{\theta}\in T^{\rm L}}\{\rho( \widetilde{\mathbf{E}}_h(\omega,\boldsymbol{\theta}))\},
\end{equation} 
where $\widetilde{\mathbf{E}}_h$ is the two-grid error operator symbol and $\rho( \widetilde{\mathbf{E}}_h)$ stands for the spectral radius of matrix $\widetilde{\mathbf{E}}_h$. Since $E_h$ contains the coarse and fine grid operators, its symbol is a $12\times 12$ matrix, including four harmonic frequencies.  

From this point onward, let us drop the subscript $h$, unless it is necessary.

The  element-wise  Vanka-type  smoother has been successfully applied  to complex-shifted Laplacian systems arising from optimal control problem \cite{HL2022shiftLaplacianPiTMG}.  Here, we consider  an element-wise additive Vanka smoother applied to \eqref{eq:shift-Laplace-stencil}. The subproblem coefficient matrix  $L_j$ in \eqref{eq:Vanka-operator} has a symmetric structure
 \begin{equation*}
L_j =\frac{\epsilon^2}{h^2}
\begin{pmatrix}
4+r &  -1 &  -1 &0 \\ 
-1 &  4+r&  0 &-1\\
-1 & 0 & 4+r &-1\\
0 & -1 & -1 &4+r 
\end{pmatrix}.
\end{equation*}

It follows that
\begin{equation} \label{eq:inverse-Li-form}
L^{-1}_j =\frac{h^2}{\epsilon^2}
\begin{pmatrix}
a & b & b &c \\
b & a& c &b\\
b & c & a &b \\
c & b & b &a 
\end{pmatrix},
\end{equation}
where 
\begin{subequations}
\label{eq:Vanka}
\begin{eqnarray}
a&=& \frac{r^2+8r+14}{(2+r)(4+r)(6+r)}, \label{eq:Vanka-a}\\
b&=& \frac{1}{(2+r)(6+r)}, \label{eq:Vanka-b}\\
c&= & \frac{2}{(2+r)(4+r)(6+r)}.\label{eq:Vanka-c}
\end{eqnarray}
\end{subequations} 
It is easy to show that $a>b>c$, which is useful for our analysis. 

Based on \eqref{eq:Vanka-operator} and \eqref{eq:inverse-Li-form},  the stencil of the element-wise Vanka smoother $M_e$ is given by
\begin{equation*} 
M_e = \frac{h^2}{4\epsilon^2} 
\begin{bmatrix}
c & 2b  &c\\
2b & 4a  &2b\\
c & 2b  &c
\end{bmatrix}.
\end{equation*} 
Using \eqref{eq:symbol-calculation-form}, we have  
\begin{eqnarray*}
\widetilde{L} &=&\frac{\epsilon^2}{h^2}(4+r-2\cos \theta_1 -2\cos \theta_2),\\
 \widetilde{M}_e &=&\frac{h^2}{\epsilon^2} (a+b\cos \theta_1 +b\cos \theta_2 +c\cos\theta_1\cos \theta_2).
\end{eqnarray*}
Let $t= \epsilon^2 (4+r-2\cos \theta_1 -2\cos \theta_2)$ and $\hat{t}=\frac{\epsilon^2}{ a+b\cos \theta_1 +b\cos \theta_2 +c\cos\theta_1\cos \theta_2)}$. Then,
\begin{equation*} 
\widetilde{\mathcal{K}}= \frac{1}{h^2}\begin{pmatrix}
t  & 0         & i 2h \sin(\theta_1/2)\\
0 &  t        & i 2h \sin(\theta_2/2)\\
-i 2h \sin(\theta_1/2)       & -i 2h \sin(\theta_2/2)    &0 
\end{pmatrix},
\end{equation*}
 and
\begin{equation*} 
\widetilde{\mathcal{M}}= \frac{1}{h^2} \begin{pmatrix}
\hat{t}  & 0         &i 2h \sin(\theta_1/2)\\
0 & \hat{t}      & i 2h \sin(\theta_2/2)\\
-i 2h \sin(\theta_1/2)       & -i 2h \sin(\theta_2/2)    &0 
\end{pmatrix},
\end{equation*}
To identify the eigenvalues of $\widetilde{\mathcal{M}}^{-1} \widetilde{\mathcal{K}}$, we first compute the determinant of $\widetilde{\mathcal{K}}-\lambda \widetilde{\mathcal{M}}$:
 \begin{eqnarray*}
 | \widetilde{\mathcal{K}}-\lambda\widetilde{\mathcal{M}}| &= &  
 \frac{1}{h^2}\begin{vmatrix}
       t-\lambda \hat{t}      & 0                      & (1-\lambda) i 2h \sin(\theta_1/2) \\
      0                     &   t-\lambda \hat{t}         & (1-\lambda) i 2h \sin(\theta_2/2) \\
 -(1-\lambda) i 2h \sin(\theta_1/2)   &(1-\lambda) i 2h \sin(\theta_2/2)  &0
    \end{vmatrix}  \\
    &=&   \frac{1}{h^2}(t-\lambda \hat{t})(1-\lambda)^2\left((i 2h \sin(\theta_1/2))^2  + (i 2h \sin(\theta_2/2) )^2\right)\\
    &=& 4\hat{t}\left((\sin(\theta_1/2))^2  + (  \sin(\theta_2/2) )^2\right)(1-\lambda)^2 (\lambda -\frac{t}{\hat{t}}).
 \end{eqnarray*}
The three eigenvalues of $\widetilde{\mathcal{M}}^{-1} \widetilde{\mathcal{K}}$  are $1, 1$ and  $\frac{t}{\hat{t}}=:\lambda^*$, where
\begin{equation}\label{eq:lambda-form-r}
\lambda^*(r;\cos \theta_1,\cos \theta_2) =(a+b\cos \theta_1 +b\cos \theta_2 +c\cos\theta_1\cos \theta_2)(4+r-2\cos \theta_1 -2\cos \theta_2).
\end{equation}

 For $\boldsymbol{\theta} \in T^{\rm H}$,  it is easy to show that 
 \begin{equation}\label{eq:cos-range}
 (\cos\theta_1,\cos\theta_2) \in \mathcal{D}=[-1,1] \times [-1,0] \bigcup [-1,0]\times [0, 1].
\end{equation}

 Next, we explore the range of $\lambda^*$ over $\boldsymbol{\theta}$ for high frequencies.

\begin{theorem}\label{thm:maxmin-and-theta}
For $\boldsymbol{\theta} \in T^{\rm H}$, 
\begin{eqnarray*}
\max_{\boldsymbol{\theta} }\lambda^*(r;\cos \theta_1,\cos \theta_2) &=& \lambda^*(r;-1,-1) =(a-2b+c)(8+r)=:d_1(r),\\
\min_{\boldsymbol{\theta} }\lambda^*(r;\cos \theta_1,\cos \theta_2) &=& \lambda^*(r;1,0) =(a+b)(2+r)=:d_2(r).
\end{eqnarray*}
\end{theorem}
\begin{proof}
For simplicity,  let $\eta_1=\cos\theta_1$ and $\eta_2=\cos \theta_2$. Then,  we rewrite \eqref{eq:lambda-form-r} as
\begin{equation*} 
 \lambda^*=\psi(\eta_1,\eta_2) =(a+b\eta_1+b\eta_2 +c\eta_1\eta_2)(4+r-2\eta_1 -2\eta_2).
\end{equation*}

We first consider the critical   point of $ \psi(\eta_1,\eta_2) $ in $\mathcal{D}$, see \eqref{eq:cos-range},  by  computing the partial derivatives of $ \psi(\eta_1,\eta_2)$, which are given by
\begin{eqnarray}
\psi'_{\eta_1} (\eta_1,\eta_2)&=&rb+4b-2a-4b\eta_1+(4c+cr-4b)\eta_2-2c\eta_2^2-4c\eta_1\eta_2=0, \label{eq:partial-psi-eta1}  \\
\psi'_{\eta_2}(\eta_1,\eta_2) &=&  rb+4b-2a-4b\eta_2+(4c+cr-4b)\eta_1-2c\eta_1^2-4c\eta_1\eta_2=0. \label{eq:partial-psi-eta2} 
\end{eqnarray}
Subtracting \eqref{eq:partial-psi-eta2}  from \eqref{eq:partial-psi-eta1}  gives 
\begin{equation}\label{eq:critical-point-solu}
(\eta_1-\eta_2)\left(2(\eta_1+\eta_2)-4-r\right) =0.
\end{equation}
It follows that $\eta_1=\eta_2$ or  $2(\eta_1+\eta_2)-4-r=0$. However, $\eta_1+\eta_2< 2$, so the latter does not have a real solution. For $\eta_1=\eta_2$, we  
replace $\eta_2$ by $\eta_1$ in \eqref{eq:partial-psi-eta1}, leading to
\begin{equation}\label{eq:solve-eta1=0}
6c\eta_1^2-(4c+cr-8b)\eta_1-(rb+4b-2a)=0.
\end{equation}
We claim that there is no real solution for \eqref{eq:solve-eta1=0} for $r>0$. We will show that the discriminant  is not positive.
We first simplify  $rb+4b-2a$. Using  \eqref{eq:Vanka-a} and \eqref{eq:Vanka-b}  gives
\begin{eqnarray*}
rb+4b-2a &=&\frac{ 4+r}{(2+r)(6+r)}-  \frac{2(r^2+8r+14)}{(2+r)(4+r)(6+r)}\\
               &=&\frac{ (4+r)^2-2(r^2+8r+14)}{(2+r) (4+r)(6+r)}\\
               &=& -\frac{1}{4+r}.
\end{eqnarray*} 

 Using \eqref{eq:Vanka-b} and \eqref{eq:Vanka-c}, the discriminant of \eqref{eq:solve-eta1=0} is
\begin{eqnarray*}
\Phi &=& (4c+cr-8b)^2+4 \cdot 6c(rb+4b-2a)\\
      &=&\left(\frac{8+2r}{(2+r)(4+r)(6+r)} - \frac{8(4+r)}{(2+r)(4+r)(6+r)} \right)^2-\frac{48}{(2+r)(4+r)(6+r)} \frac{1}{4+r}\\
      &=&\left (\frac{-6}{(2+r)(6+r)}\right)^2-\frac{48}{(2+r)(4+r)^2(6+r)}\\
      &=& \frac{12}{(2+r)(6+r)}\left(\frac{3}{(2+r)(6+r)}-\frac{4}{(4+r)^2}\right)\\
      &=& \frac{-12r(r+8)}{(2+r)^2(4+r)^2(6+r)^2}\leq 0.
\end{eqnarray*}
The case $r=0$ has been discussed  \cite{CH2021addVanka}, where $\psi'_{\eta_1} (\eta_1,\eta_2)=\psi'_{\eta_2} (\eta_1,\eta_2)=0$ gives $(\eta_1,\eta_2)=(-1,-1)$ the boundary  point of  $\mathcal{D}$, and $\lambda^*_{\rm max}=\frac{4}{3}$.  When $r>0$,  \eqref{eq:critical-point-solu} has no real solution and $\psi(\eta_1,\eta_2)$ cannot have extreme values at interior of $\mathcal{D}$.  This means that we only need to find the   extreme values of $\psi(\eta_1,\eta_2)$ at the boundary of  $\mathcal{D}$, see \eqref{eq:cos-range}. To do this, we split the boundary of $\mathcal{D}$ as follows: 
\begin{eqnarray*}
\partial \mathcal{D}_1 &=&\{-1\} \times [-1,1], \\
\partial \mathcal{D}_2 &=& [-1,1] \times \{-1\}, \\
\partial \mathcal{D}_3 &=& \{1\} \times  [-1,0], \\
\partial  \mathcal{D}_4 &=& [0,1]\times \{0\}, \\
\partial \mathcal{D}_5 &=& \{0\} \times  [0,1], \\
 \partial \mathcal{D}_6&=& [-1,0] \times \{1\}.  
\end{eqnarray*} 
 
Due to the symmetry of $\psi(\eta_1,\eta_2)$, that is $\psi(\eta_1,\eta_2)=\psi(\eta_2,\eta_1)$, we only need to find the extreme values of $ \psi(\eta_1,\eta_2) $ at $\partial \mathcal{D}_1, \partial \mathcal{D}_3$  and $\partial \mathcal{D}_4$. We present below the  results.

\begin{enumerate}
\item For $(\eta_1,\eta_2) \in \partial \mathcal{D}_1$,
\begin{equation}\label{eq:psi-case1}
\psi(\eta_1,\eta_2) =\psi(-1,\eta_2) =(a-b+b\eta_2-c\eta_2)(6+r-2\eta_2).
\end{equation} 
Note that  the two roots of the quadratic form \eqref{eq:psi-case1} are $\frac{6+r}{2}$ and $\frac{a-b}{c-b}$.
Using \eqref{eq:Vanka}, we have
\begin{equation*}
\frac{a-b}{c-b}= -(5+r).
\end{equation*} 
Thus,  the axis of symmetry is $\eta_2=\frac{(6+r)/2-(5+r)}{2}=-1-\frac{r}{4}\leq -1$. Using the fact that $a>b>c$, see \eqref{eq:Vanka-a}, \eqref{eq:Vanka-b}, and \eqref{eq:Vanka-c}, the quadratic function opens downward.  Therefore, 
 the maximum and minimum of $\psi(-1,\eta_2) $ for $\eta_2 \in[-1,1]$ are  
\begin{align}
\begin{aligned}
\psi(-1,\eta_2) _{\rm max} &=\psi(-1,-1)  =(a-2b+c)(8+r),\label{eq:case1-max}\\
\psi(-1,\eta_2) _{\rm min} &=\psi(-1,1)  = (a-c)(4+r). 
\end{aligned}
\end{align}
\item For $(\eta_1,\eta_2) \in \partial \mathcal{D}_3$,
\begin{equation}\label{eq:psi-case2}
\psi(\eta_1,\eta_2) =\psi(1,\eta_2) =(a+b+b\eta_2+c\eta_2)(2+r-2\eta_2).
\end{equation} 
The two roots of quadratic form \eqref{eq:psi-case2} are $\frac{2+r}{2}$ and $-\frac{a+b}{b+c}$.
Using \eqref{eq:Vanka},  we have
\begin{equation*}
-\frac{a+b}{b+c}= -(3+r).
\end{equation*} 
Thus,  the axis of symmetry is $\eta_2=\frac{(2+r)/2-(3+r)}{2}=-1-\frac{r}{4}\leq -1$. Using the fact that $a>b>c$, the quadratic function opens downward.  It follows that
for $\eta_2 \in[-1,0]$, the maximum and minimum of $\psi(1,\eta_2) $ are given by
\begin{eqnarray*}
\psi(1,\eta_2) _{\rm max} &=&\psi(1,-1)  =(a-c)(4+r),\\
\psi(1,\eta_2) _{\rm min} &=&\psi(1,0)  = (a+b)(2+r).
\end{eqnarray*}

\item For $(\eta_1,\eta_2) \in \partial \mathcal{D}_4$,
\begin{equation*} 
\psi(\eta_1,\eta_2) =\psi(\eta_1,0) =(a+b\eta_1)(4+r-2\eta_1).
\end{equation*} 
Note that  the two roots of quadratic form \eqref{eq:psi-case2} are $\frac{4+r}{2}$ and $-\frac{a}{b}$.
Using \eqref{eq:Vanka-a} and \eqref{eq:Vanka-b}, we have
\begin{equation*}
-\frac{a}{b}= -(4+r)+\frac{2}{4+r}.
\end{equation*} 
Thus,  the axis of symmetry is $\eta_1=\frac{(4+r)/2-(4+r)+\frac{2}{4+r}}{2}=-1-\frac{r}{4}+\frac{1}{4+r}< 0$.  Thus, the maximum and minimum of $\psi(\eta_1,0) $ for $\eta_1 \in[0,1]$ are  
 
\begin{eqnarray}
\psi(\eta_1,0) _{\rm max} &=&\psi(0,0)  =a(4+r),\label{eq:case3-max}\\
\psi(\eta_1,0) _{\rm min} &=&\psi(0,1)  = (a+b)(2+r). \nonumber
\end{eqnarray}
\end{enumerate}

Based on the above discussions,  the minimum of $\psi(\eta_1,\eta_2) $ over $\partial \mathcal{D}$ is
\begin{equation*}
\psi(\eta_1,\eta_2) _{\rm min} =\psi(1,0)=\psi(0,1)  = (a+b)(2+r).
\end{equation*} 
Next, we compare $\psi(-1,-1)$ (see \eqref{eq:case1-max}) and  $\psi(0,0)$ (see \eqref{eq:case3-max})  to determine the maximum.
Using   \eqref{eq:Vanka},  we have 
\begin{eqnarray*}
\psi(-1,-1)- \psi(0,0) &=&(a-2b+c)(8+r)-a(4+r)\\
                              &=& 4a+(c-2b)(8+r)\\
                              & =&\frac{4(r^2+8r+14)}{(2+r)(4+r)(6+r)} - \frac{2-2(4+r)}{(2+r)(4+r)(6+r)}(8+r)\\
                              &=& \frac{2r^2+10r+8}{(2+r)(4+r)(6+r)}>0.
\end{eqnarray*}
It follows that the maximum of $\psi(\eta_1,\eta_2)$ is given by
 \begin{equation*}
\psi(\eta_1,\eta_2) _{\rm max} =\psi(-1,-1) =(a-2b+c)(8+r).
\end{equation*} 
Thus, for $(\eta_1,\eta_2) \in \mathcal{D}$, the maximum and minimum of $\psi(\eta_1,\eta_2)$  are  $\psi(-1,-1)$ and $\psi(1,0)=\psi(0,1)$, respectively. 
\end{proof}

Based on the results in Theorem \ref{thm:maxmin-and-theta}, we can further estimate the range of  extreme values of $\lambda^*$, which plays an important role in determining the optimal smoothing factor for V-BSR.

\begin{theorem}\label{thm:form-d1-d2}
Suppose $r\in[0,\infty)$. Then,
\begin{equation}\label{eq:explicit-form-low-up-bound}
d_1(r)=\frac{8+r}{6+r}, \quad d_2(r)=\frac{3+r}{4+r}. 
\end{equation}
Furthermore, 
\begin{eqnarray*}
1< &d_1(r)& \leq \frac{4}{3}, \\
 \frac{3}{4}\leq &d_2(r)& <1.
\end{eqnarray*} 
\end{theorem}
\begin{proof}
Using \eqref{eq:Vanka},  we simplify $d_1(r)$ as follows:
\begin{eqnarray*}
d_1(r) &=&(a-2b+c)(8+r) \\
          &=& \frac{r^2+8r+14-2(4+r)+2}{(2+r)(4+r)(6+r)}(8+r)\\
          &=& \frac{8+r}{6+r}.
\end{eqnarray*} 
Since $d_1(r)$ is a decreasing function of $r$, $\max_r d_1(r)=d_1(0)=\frac{4}{3}$.

Using \eqref{eq:Vanka-a} and \eqref{eq:Vanka-b},  we  have
\begin{eqnarray*}
d_2(r) &=&(a+b)(2+r) \\
          &=& \frac{r^2+8r+14+4+r}{(2+r)(4+r)(6+r)} (2+r)\\
          &=&\frac{r+3}{r+4}.
\end{eqnarray*}
Since $d_2(r)$ is an increasing function of $r$, $\min_r d_2(r)=d_2(0)=\frac{3}{4}$. 
\end{proof}

Now, we are able to derive the optimal smoothing  factor for  V-BSR for the Stokes-Darcy Brinkman problems.
\begin{theorem}\label{thm:opt-mu-omega}
For the V-BSR relaxation scheme \eqref{eq:BSR-relax-scheme} the optimal smoothing factor is given by
\begin{equation}\label{eq:opt-mu-form-r}
\mu_{\rm opt}(r)=\min_{\omega} \max_{\boldsymbol{\theta} \in T^{\rm H}} \{|1-\omega|, |1-\omega \lambda^*|\} = \frac{3r+14}{2r^2+21r+50},
\end{equation}
provided that  
\begin{equation}\label{eq:opt-omega}
\omega =\omega_{\rm opt}= \frac{2r^2+20r+48}{2r^2+21r+50}.
\end{equation}
Moreover,
\begin{equation*}
\mu_{\rm opt}(r)\leq \frac{7}{25}=0.28.
\end{equation*} 

\end{theorem}
 \begin{proof}
From Theorem \ref{thm:maxmin-and-theta} and \eqref{eq:explicit-form-low-up-bound}, we know  that
 \begin{equation}\label{eq:lambda*-two-roots}
  \max_{\boldsymbol{\theta} \in T^{\rm H}} \{|1-\omega|, |1-\omega \lambda^*|\} =\max\{|1-\omega d_2(r)|, |1-\omega d_1(r)| \}.
  \end{equation}
  To minimize $\max_{\boldsymbol{\theta} \in T^{\rm H}} \{|1-\omega \lambda^*|\}$, we require
  \begin{equation*}
|1-\omega d_2(r)|= |1-\omega d_1(r)|,
\end{equation*} 
which gives $\omega_{\rm opt}(r) =\frac{2}{d_1(r)+d_2(r)}$.  Using Theorem \ref{thm:form-d1-d2}, we obtain 
\begin{equation*}
 \omega_{\rm opt}(r) =\frac{2}{d_1(r)+d_2(r)} = \frac{2r^2+20r+48}{2r^2+21r+50}
\end{equation*} 
and
\begin{equation*}
\mu_{\rm opt}(r) =\frac{d_1(r)-d_2(r)}{d_1(r)+d_2(r)} = \frac{3r+14}{2r^2+21r+50}\leq \frac{7}{25}.
\end{equation*} 
 \end{proof}
 
 \begin{remark}
 When $r=0$, Theorem \ref{thm:opt-mu-omega} is consistent with the existing results \cite{CH2021addVanka},  which amount to applying an element-wise Vanka smoother  to  the Poisson equation. 
 \end{remark}
 \begin{proposition}\label{prop:optimal-mu-decrease}
 $\mu_{\rm opt}(r)$ given by \eqref{eq:opt-mu-form-r} is a decreasing function of $r$.  
 \end{proposition}
 \begin{proof}
The derivative of $\mu_{\rm opt}$ is  given by
 \begin{equation*}
 \mu'_{\rm opt} =  \frac{-2(r^2+28r+72)}{(2r^2+21r+50)^2}<0.
\end{equation*} 
 This suggests that when $r$ increases, the optimal smoothing factor decreases. 
 \end{proof}

Let us look at the optimal parameter, \eqref{eq:opt-omega}. It can be shown that
 \begin{equation*}
 \omega'_{\rm opt}(r) =  \frac{2(r^2+4r-4)}{(2r^2+21r+50)^2}.
\end{equation*} 
When $r \in[0, 2\sqrt{2}-2]$, $\omega_{\rm opt}(r)$ is decreasing and  when $r \in[2\sqrt{2}-2, \infty)$, $\omega_{\rm opt}(r)$ is increasing. It follows that
 \begin{equation}\label{eq:estimate-opt-omega}
(\omega_{\rm opt}(r))_{\rm min}=\omega_{\rm opt}(2\sqrt{2}-2) = \frac{(\sqrt{2}-1)(4\sqrt{2}+16)+24}{(\sqrt{2}-1)(4\sqrt{2}+17)+25}\approx 0.959 \leq \omega_{\rm opt}(r) <1.
\end{equation} 
 Thus, for simplicity, if we take $\omega=1$, then \eqref{eq:lambda*-two-roots} gives
 \begin{equation}\label{eq:omega-one-mu}
\mu(\omega=1)=\max_{\omega=1}\{|1-\omega d_2(r)|, |1-\omega d_1(r)| \}=\frac{2}{6+r}\leq \frac{1}{3}.
\end{equation} 

In practice, we can consider $\omega=1$.  In multigrid, for fixed $\epsilon$, in each level, the relaxation schemes has a different convergence speed in each level, which can be computed from \eqref{eq:omega-one-mu}.   However, note that \eqref{eq:omega-one-mu} is a decreasing function of $r=\frac{h^2}{\epsilon^2}$ or $h$. This means that at the coarse level, the relaxation scheme has a smaller convergence speed compared with that of the fine level.  
 
 \section{Numerical experiments}\label{sec:numerical}
 In this section, we first compute LFA two-grid convergence factors using two choices of the damping parameter, that is, $\omega=1$ and $\omega=\omega_{\rm opt}$, then we report V-cycle multigrid results for different values of the physical parameter $\epsilon$.
 
\subsection{LFA prediction}
We compute the two-grid LFA convergence factor \eqref{eq:LFA-two-grid-convergence-factor},  using   $h=1/64,\  \omega=1$ and  optimal $\omega$, see \eqref{eq:opt-omega}, derived from optimizing LFA smoothing factors for different $\epsilon$.   From Table \ref{tab:LFA-results-h64-omega} we see a strong agreement between two-grid convergence factors $\rho_h(1)$  and  the LFA smoothing factors.  Moreover,  the convergence factors  for optimal $\omega$ are slightly better than those for $\omega=1$, which is reasonable  since the optimal $\omega$, see \eqref{eq:estimate-opt-omega}, is very close to $1$.  From our smoothing analysis, we know that even though  the smoothing factor is dependent on $h$ and $\alpha$,  the upper bound on the smoothing factor is $\frac{1}{3}$.  This is also confirmed by our two-grid LFA convergence factor $\rho_h(1)$ in Table \ref{tab:LFA-results-h64-omega}.

 \begin{table}[H]
 \caption{Two-grid LFA convergence factors, $\rho_h(\nu)$ with $h=1/64$ and different choices of $\omega$.}
\centering
\begin{tabular}{lccccc}
\hline
$\epsilon,\omega=1$                      & $\mu_{\rm opt}$        &$\rho_h(1)$      &$\rho_h(2)$     &$\rho_h(3)$    & $\rho_h(4)$  \\ \hline
 
$1$                                        & 0.333                          & 0.333           &   0.119                   & 0.054                  &0.043    \\
$ 2^{-2} $                                  &0.333                         & 0.333                & 0.119                 & 0.054        & 0.042     \\
$2^{-4}$                                  &0.330                        &  0.330                &0.115                  &0.052           &0.040     \\
$2^{-6}$                                   &0.286                        & 0.286               & 0.082                &  0.023          &   0.012   \\
$2^{-8}$                                   &  0.091                       & 0.091              & 0.008                & 0.001            &0.000     \\
\hline
$1, \omega_{\rm opt}$                   &0.280                         &0.280                   & 0.096              &  0.056          &  0.044     \\

$2^{-2}$         & 0.280                        &0.280                 &0.096                 & 0.056            &0.044     \\

$2^{-4}$         & 0.276                       & 0.276                & 0.093              & 0.055         &0.042      \\

$2^{-6}$          & 0.233                       &0.233               &  0.057             &   0.026          &    0.014      \\

$2^{-8}$          & 0.069                        & 0.069                 & 0.005                & 0.000          &  0.000         \\
\hline
\end{tabular}\label{tab:LFA-results-h64-omega}
\end{table}

To illustrate how the smoothing factor changes as a function of $r$, we plot $\mu_{\rm opt}$ defined in \eqref{eq:opt-mu-form-r} and  $\mu(\omega=1)$ in \eqref{eq:omega-one-mu}  as   functions of $r$ in  Figure \ref{fig: mu-vs-r}.   It is evident that when $r$ increases, the smoothing factor decreases and approaches zero, and $\mu(\omega=1)$ tends towards  $\mu_{\rm opt}$.

\begin{figure}[H]
\centering
\includegraphics[width=0.6\textwidth]{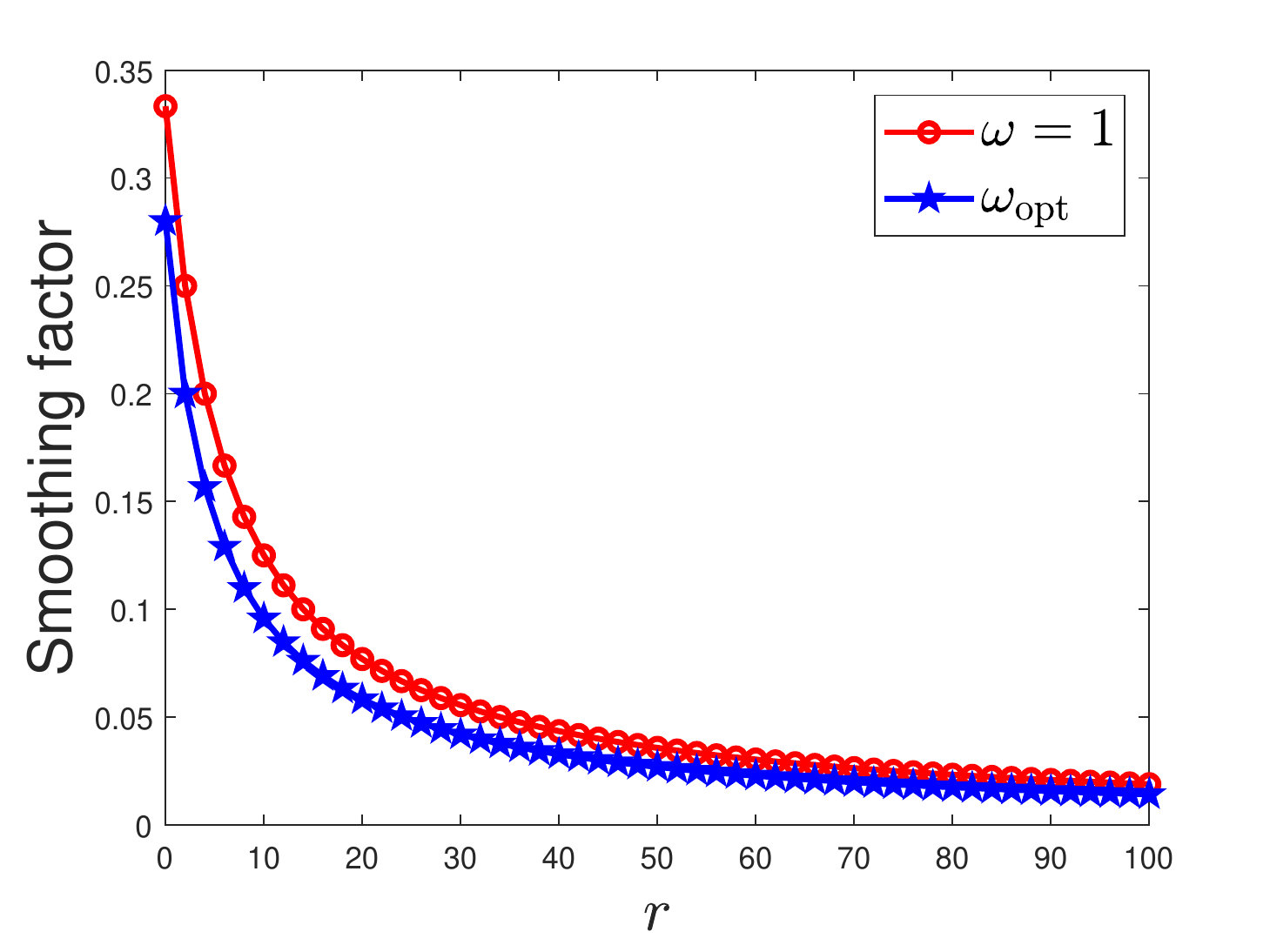}
 \caption{Smoothing  factors with optimal $\omega$ and $\omega=1$.}\label{fig: mu-vs-r}
\end{figure}

 \subsection{Multigrid performance}
 
 Consider the model problems \eqref{eq:SDB}  on the unit square domain $[0,1]\times[0,1]$ with an exact solution \cite[Section 5]{sun2019stability},   
 and given by 
\begin{align*}
 u(x,y) &=\pi \sin^2(\pi x) \sin(2\pi y),\\
 v(x,y) &=-\pi \sin(2\pi x) \sin^2(\pi y),\\
 p(x,y) &=\sin(\pi y) -\frac{2}{\pi},
\end{align*} 
with $g=0$. The source term is computed  via $\boldsymbol{f}=(f_1,f_2)=- \epsilon^2 \Delta\boldsymbol{u} +\boldsymbol{u} +  \nabla  p$,  and it is
\begin{align*}
f_1 &=(4\pi^3\epsilon^2+\pi)\sin^2(\pi x) \sin(2\pi y)-2\pi^3\epsilon^2\cos(2\pi x)\sin(2\pi y), \\
f_2& = -(4\pi^3\epsilon^3+\pi)\sin(2\pi x)\sin^2(\pi y)+2\pi^3\epsilon^2\sin(2\pi x)\cos(2\pi y)+\pi \cos(\pi y).
\end{align*}

To validate our theoretical LFA predictions, we compute the actual multigrid convergence factors by
\begin{equation*}
\hat{\rho}^{(k)}_h=\left( \frac{||r_k||}{||r_0||}\right)^{1/k},
\end{equation*}
where $r_k=b_h-\mathcal{K}_h\boldsymbol{z}_k$ is the residual  and $\boldsymbol{z}_k$ is the $k$-th multigrid iteration. The initial guess is chosen randomly.  In our test, we report  $\hat{\rho}^{(k)}_h=:\hat{\rho}_h $ with  the smallest $k$ such that $||r_k||/|r_0|\leq 10^{-10}$.  

As mentioned before, computing the exact solution of the Schur complement system \eqref{eq:solution-schur-complement} is expensive. For our multigrid tests, we apply a few weighted ($\omega_J$) Jacobi iterations to the Schur complement system. We choose $\omega_J=0.8$ which seems more robust to $\epsilon$. The number of Jacobi iterations is set as  three. 

\subsubsection{Two-grid results}
We first report actual two-grid convergence factor  using  $h=1/64$ and  three Jacobi iterations for solving the Schur complement.    Table \ref{tab:TG-omega1} shows that two-grid actual performance using $\omega=1$ matches with  the LFA predictions in  Table \ref{tab:LFA-results-h64-omega}, except for a small difference for a $\epsilon =2^{-8}$, which might suggest that more iterations are needed for the Schur complement system. However, due to the satisfactory convergence factor of the actual performance, we do not further explore this.  Using optimal $\omega$, Table \ref{tab:TG-omega2} shows that the actual two-grid performance matches two-grid LFA predictions reported in Table   \ref{tab:LFA-results-h64-omega}, except for $\epsilon=2^{-6}, 2^{-8}$. Again,  the measured convergence factor is satisfactory, and there is no need to consider more Jacobi iterations for the Schur complement system.

 \begin{table}[H]
 \caption{Two-grid measured convergence factor, $\hat{\rho}_h(\nu)$, using three Jacobi iterations for solving the Schur complement,  $h=1/64$ and   $\omega=1$.}
\centering
\begin{tabular}{lcccc }
\hline
$  \epsilon, \omega=1 $                               &$\hat{\rho}_h(1)$      &$\hat{\rho}_h(2)$     &$\hat{\rho}_h(3)$    & $\hat{\rho}_h(4)$  \\ \hline
 
$1$                                                             & 0.319                  &0.111            &0.033       &   0.023     \\
$2^{-2} $                                                     &0.317                  & 0.109             &0.033         &  0.023   \\
$2^{-4} $                                                    & 0.300                 &  0.094            & 0.029       &  0.021    \\
$2^{-6}$                                                     & 0.209                 &0.047             &0.023       & 0.015       \\
$2^{-8} $                                                     &0.145                   & 0.035             &0.020         & 0.015    \\
\hline
 \end{tabular}\label{tab:TG-omega1}
\end{table}

 \begin{table}[H]
 \caption{Two-grid measured convergence facto, $\hat{\rho}_h(\nu)$, using three Jacobi iterations for solving the Schur complement, $h=1/64$ and  $\omega_{\rm opt}$, see \eqref{eq:opt-omega}.}
\centering
\begin{tabular}{lcccc }
\hline
$\epsilon, \omega_{\rm opt}$                               &$\hat{\rho}_h(1)$      &$\hat{\rho}_h(2)$     &$\hat{\rho}_h(3)$    & $\hat{\rho}_h(4)$  \\ \hline
 
$1$                                                             & 0.266                & 0.082            &0.030       &   0.023      \\
$2^{-2} $                                                    &  0.264             &0.080              & 0.030         &  0.024    \\
$2^{-4}$                                                    &  0.248               &0.068           &0.029            & 0.023   \\
$2^{-6}$                                                     &0.163                &  0.044           & 0.024          &  0.016   \\
$2^{-8}$                                                     & 0.165                &0.040               & 0.019         &0.015   \\
\hline
 \end{tabular}\label{tab:TG-omega2}
\end{table}

\subsubsection{V(1,1)-cycle results}
A two-grid method is computationally costly since  we have to  solve the coarse problem directly and if the initial mesh is fine, then the next coarser mesh may give rise to a large problem as well.  In practice, deeply-nested W-cycle and V-cycle are preferred.  We now explore the V(1,1)-cycle multigrid methods with two choices of $\omega$ and varying values of the physical parameter $\epsilon$.  In order to study the sensitivity of solving the Schur complement system, we consider  one, two, and three Jacobi iterations for Schur complement system.  We consider different $n\times n$ finest meshgrids, where $n=32,  64,128, 256$.

{\bf One iteration for Schur complement system:} We first report the iteration counts for  V(1,1)-cycle multigrid methods using one iteration of Jacobi relaxation for the Schur complement system in Table \ref{tab:Itn-Schur-Jacobi-number-one} to achieve the tolerance $||r_k||/|r_0|\leq 10^{-10}$. We see that using $\omega=1$ and optimal $\omega$ give similar performance. When $\epsilon =2^{-6}, 2^{-8}$, the iteration count increase dramatically.
To mitigate the effect of this degradation, we will consider two or three Jacobi iterations for the Schur complement system.

 \begin{table}[H]
 \caption{Iteration accounts for V(1,1)-cycle multigrid with one Jacobi iteration for solving the Schur complement.}
\centering
\begin{tabular}{lcccc }
\hline
$\epsilon, \omega=1$                               &$n=32$      &$n=64$     &$n=128$    & $n=256$  \\ 
 
$1$                                                             &  13                &   13               &13        &  15            \\
$2^{-2} $                                                   &12                  & 13                 & 13       &  14           \\
$2^{-4}$                                                    &11                 &  11               &  12          &  12         \\
$2^{-6}$                                                    &23                  & 18                 &13        &  11           \\
$2^{-8}$                                                     & 50                 &   50                &47      &  34            \\
\hline
$\epsilon, \omega_{\rm opt}$                   &$n=32$      &$n=64$     &$n=128$    & $n=256$  \\ 
 
$1$                                                             & 12                   &12                & 12         & 15           \\
$2^{-2} $                                                    & 12                  & 12                  &  12       &  14         \\
$2^{-4}$                                                    & 11                  & 11                  &11          &  11        \\
$2^{-6}$                                                    &  26                  &19               &  14       &  12          \\
$2^{-8}$                                                    &50                  &50                 &50       & 38            \\
\hline
 \end{tabular}\label{tab:Itn-Schur-Jacobi-number-one}
\end{table}

{\bf Two iterations for Schur complement system:} We report the convergence history of the relative residual norm $\frac{||r_k||}{||r_0||}$  as a function of the number of V(1,1)-cycles using two Jacobi iterations for Schur complement system.  Figure \ref{fig:V-vs-J2-eps0} reports the results for $\epsilon=1$. We see that using optimal $\omega$ takes 12  V(1,1)-cycle iterations to achieve the stopping tolerance and it takes 13 iterations for $\omega=1$. The convergence behavior is independent of meshsize $h$.  A similar performance is seen for  $\epsilon=2^{-2}, 2^{-4}, 2^{-6}, 2^{-8}$ in Figures  \ref{fig:V-vs-J2-eps2}, \ref{fig:V-vs-J2-eps4},  \ref{fig:V-vs-J2-eps6} and  \ref{fig:V-vs-J2-eps8}.   Observe that for smaller values of $\epsilon$, the iteration count does not increase.   Using optimal $\omega$  has   one iteration number  fewer than that  of $\omega=1$.  Thus, it is simple and reasonable to use $\omega=1$ in practice.

\begin{figure}[h!] 
\centering
\includegraphics[width=0.49\textwidth]{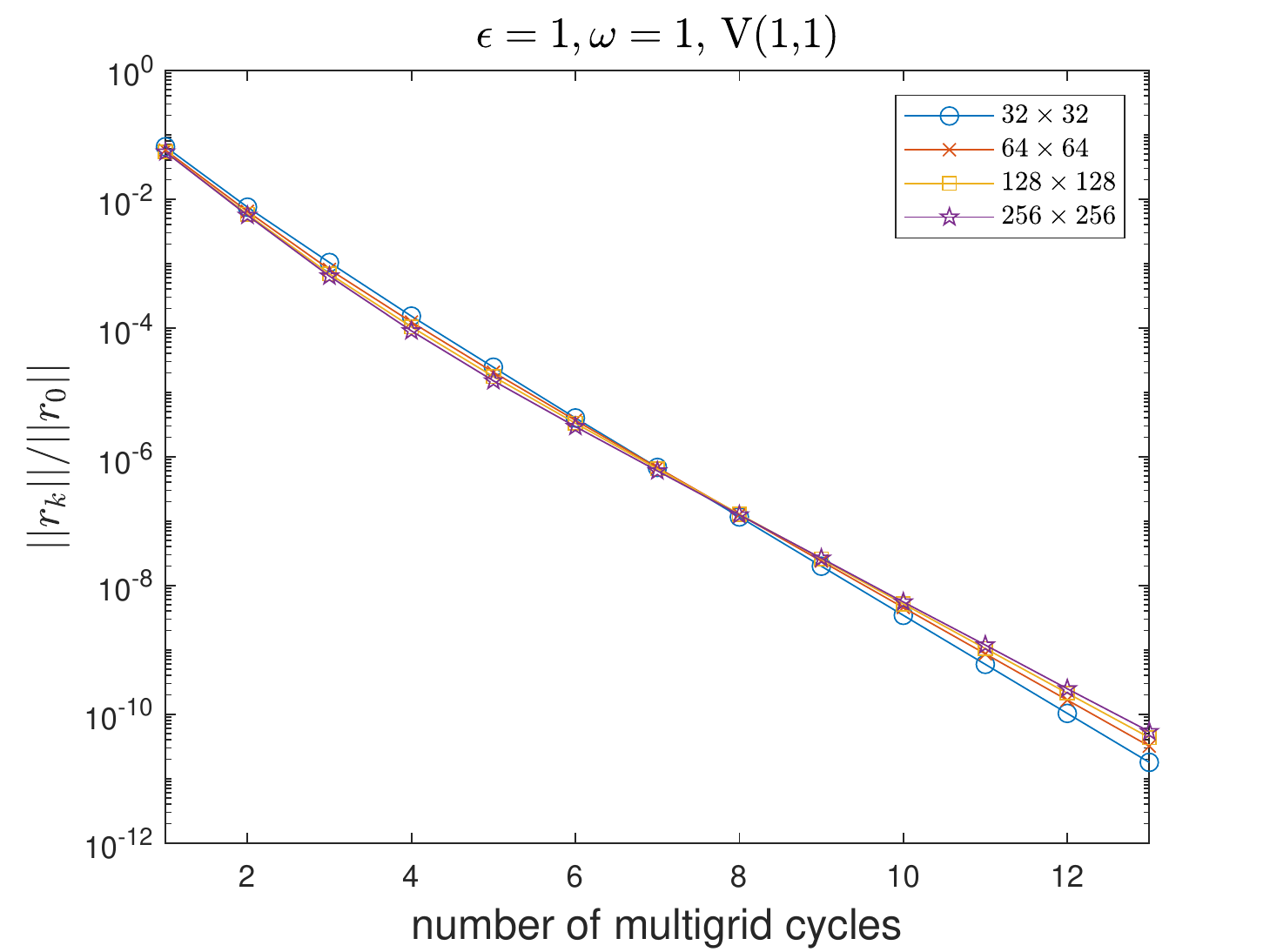}
\includegraphics[width=0.49\textwidth]{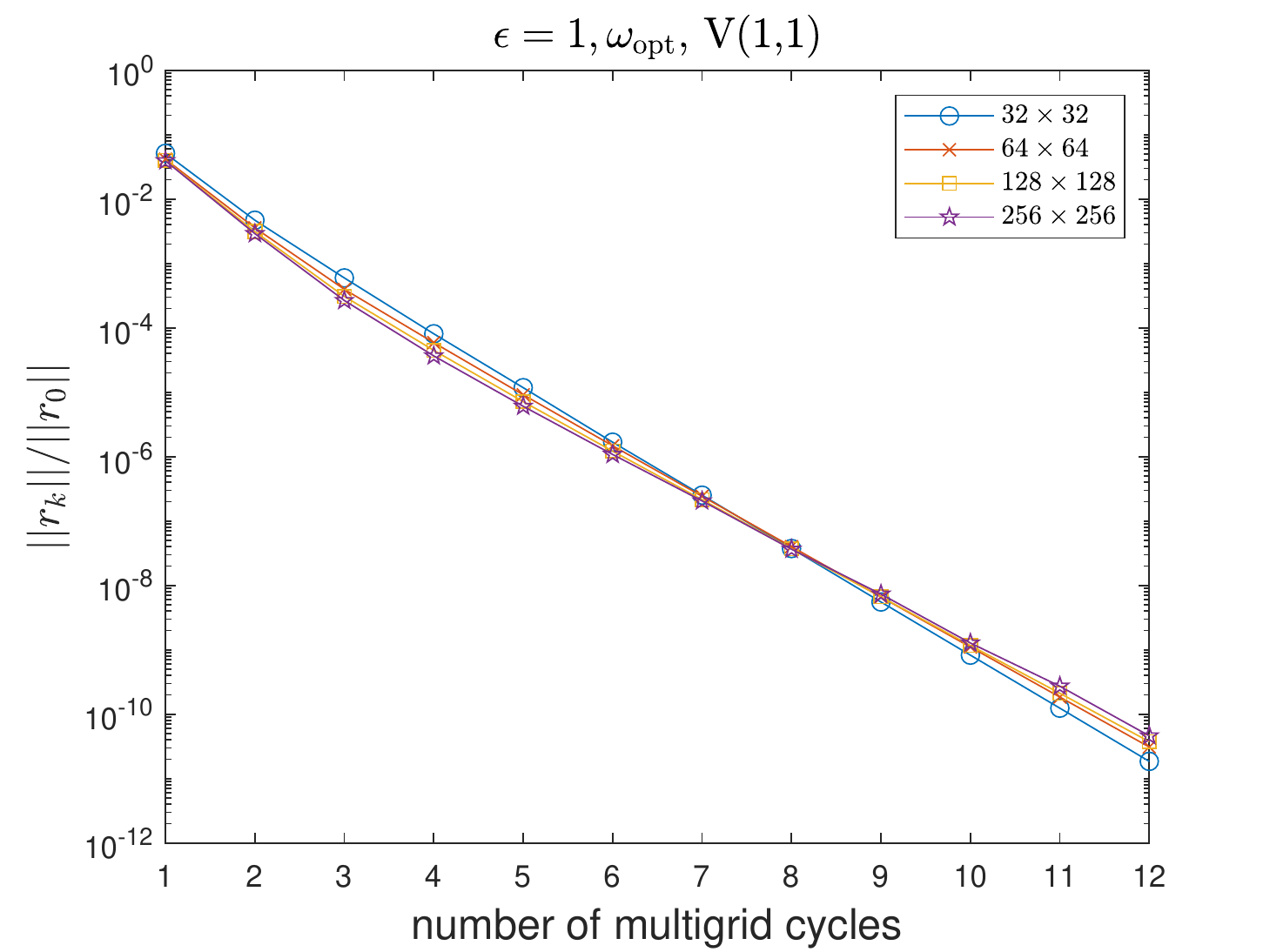}
 \caption{Convergence history: Number of iterations versus relative residual of V(1,1)-cycle with $\epsilon=1$ and two Jacobi iterations for Schur complement system (left $\omega=1$ and right  optimal $\omega$).} \label{fig:V-vs-J2-eps0}
\end{figure}

\begin{figure}[h!]
\centering
\includegraphics[width=0.49\textwidth]{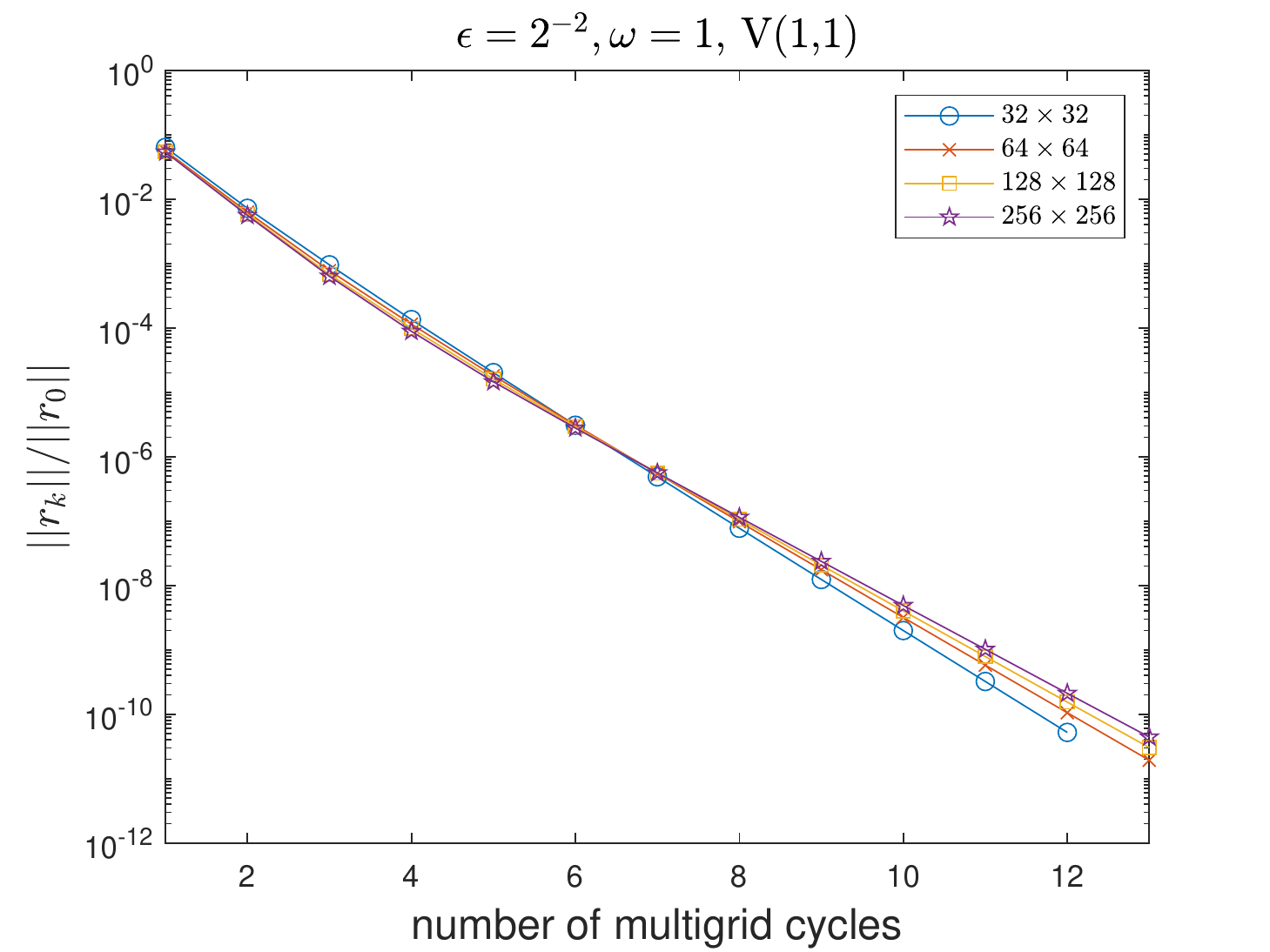}
\includegraphics[width=0.49\textwidth]{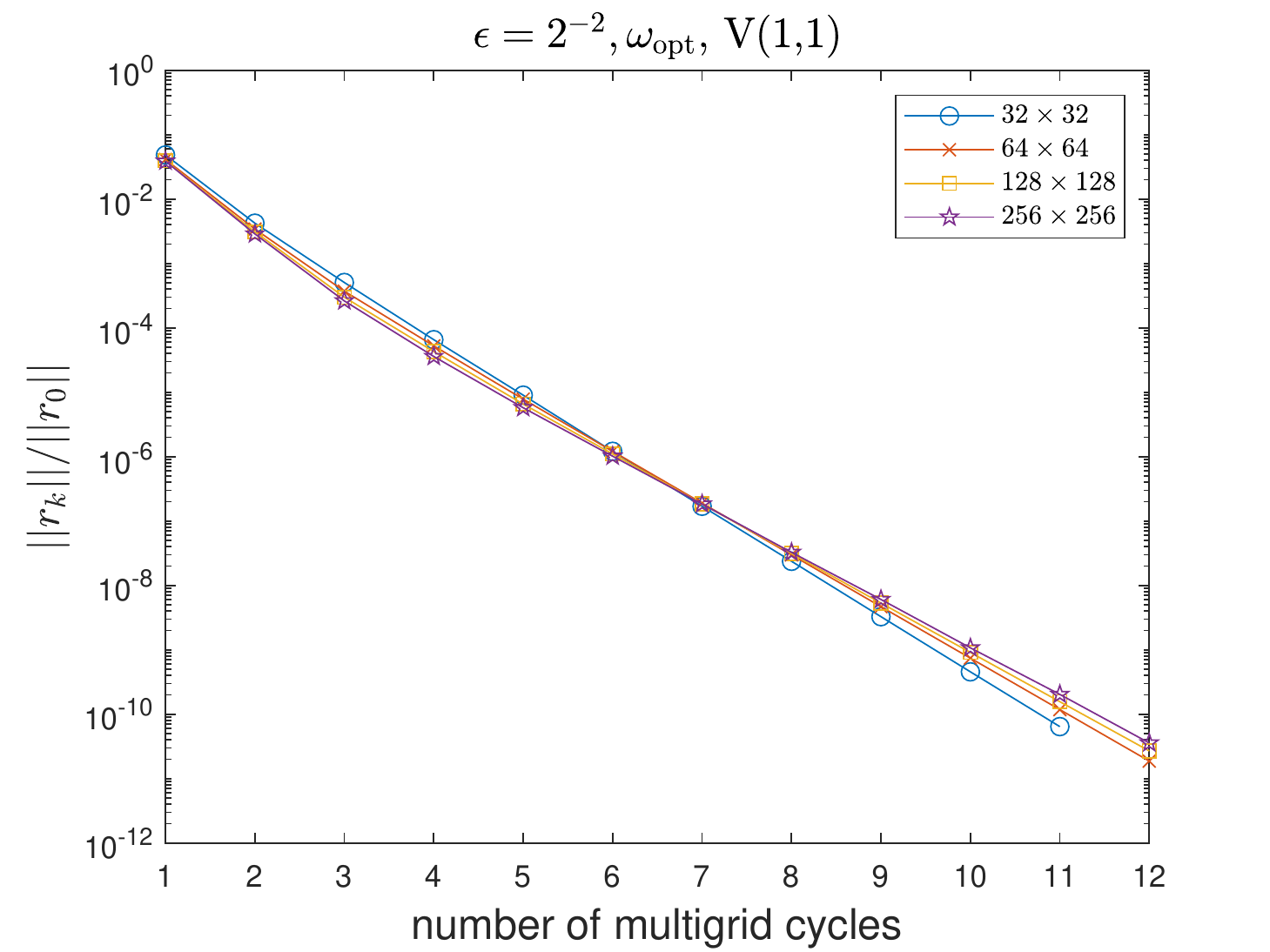}
 \caption{Convergence history: Number of iterations versus relative residual  of V(1,1)-cycle with $\epsilon=2^{-2}$ and two Jacobi iterations for Schur complement system (left $\omega=1$ and right  optimal $\omega$).} \label{fig:V-vs-J2-eps2}
\end{figure}

\begin{figure}[h!] 
\centering
\includegraphics[width=0.49\textwidth]{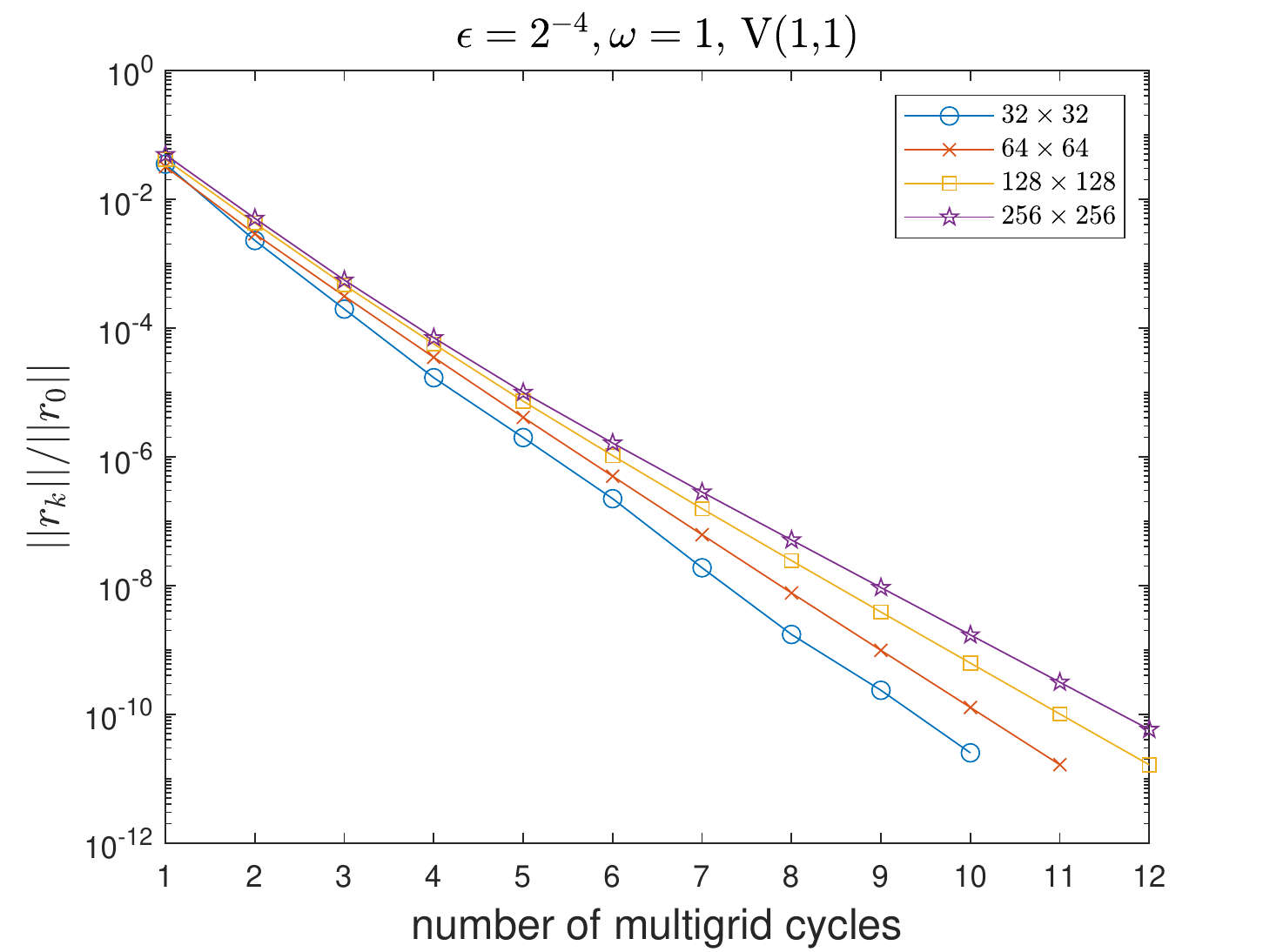}
\includegraphics[width=0.49\textwidth]{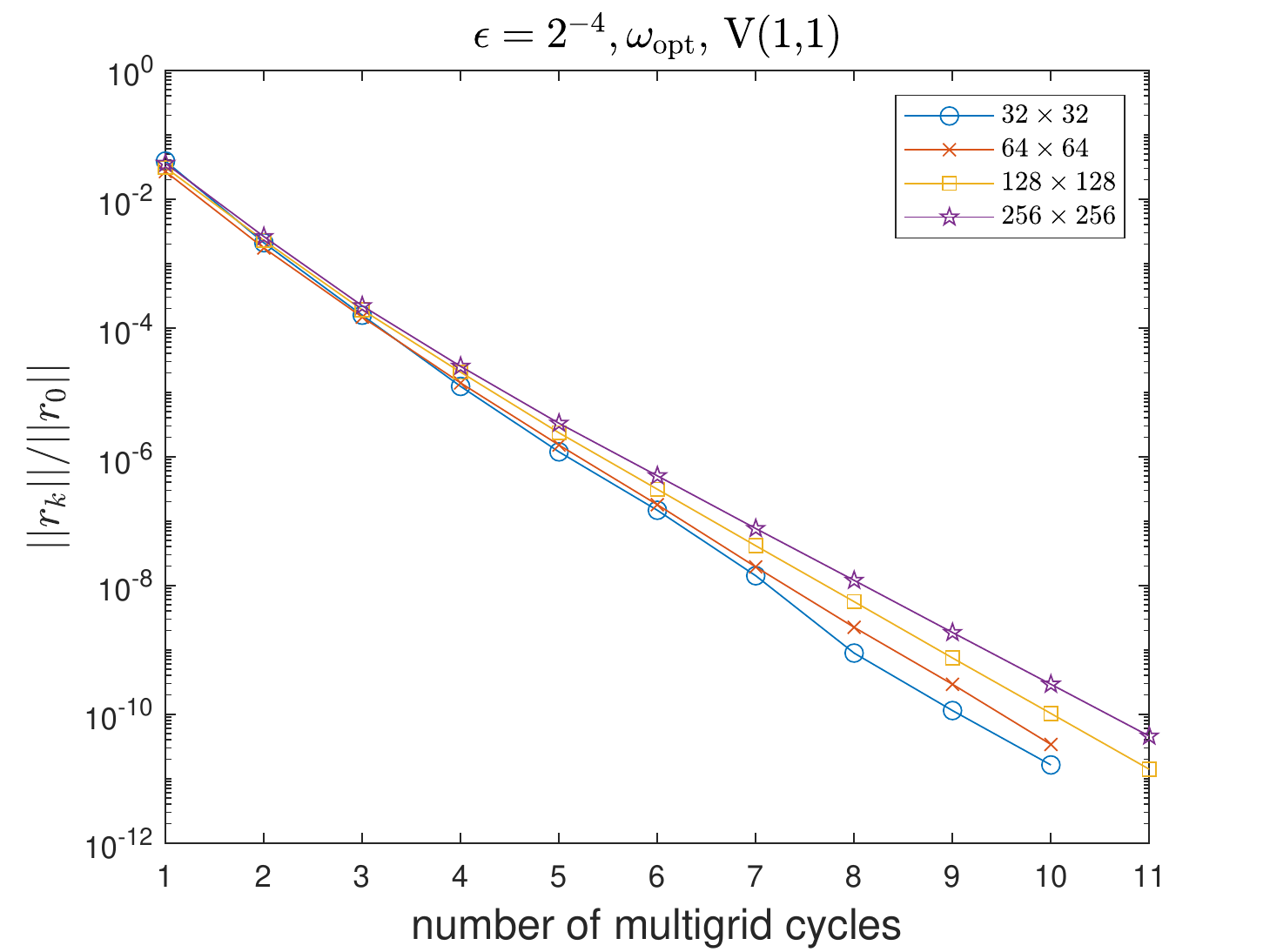}
 \caption{Convergence history: Number of iterations versus relative residual  of V(1,1)-cycle with $\epsilon=2^{-4}$ and two Jacobi iterations for Schur complement system (left $\omega=1$ and right  optimal $\omega$).} \label{fig:V-vs-J2-eps4}
\end{figure}

\begin{figure}[h!]
\centering
\includegraphics[width=0.49\textwidth]{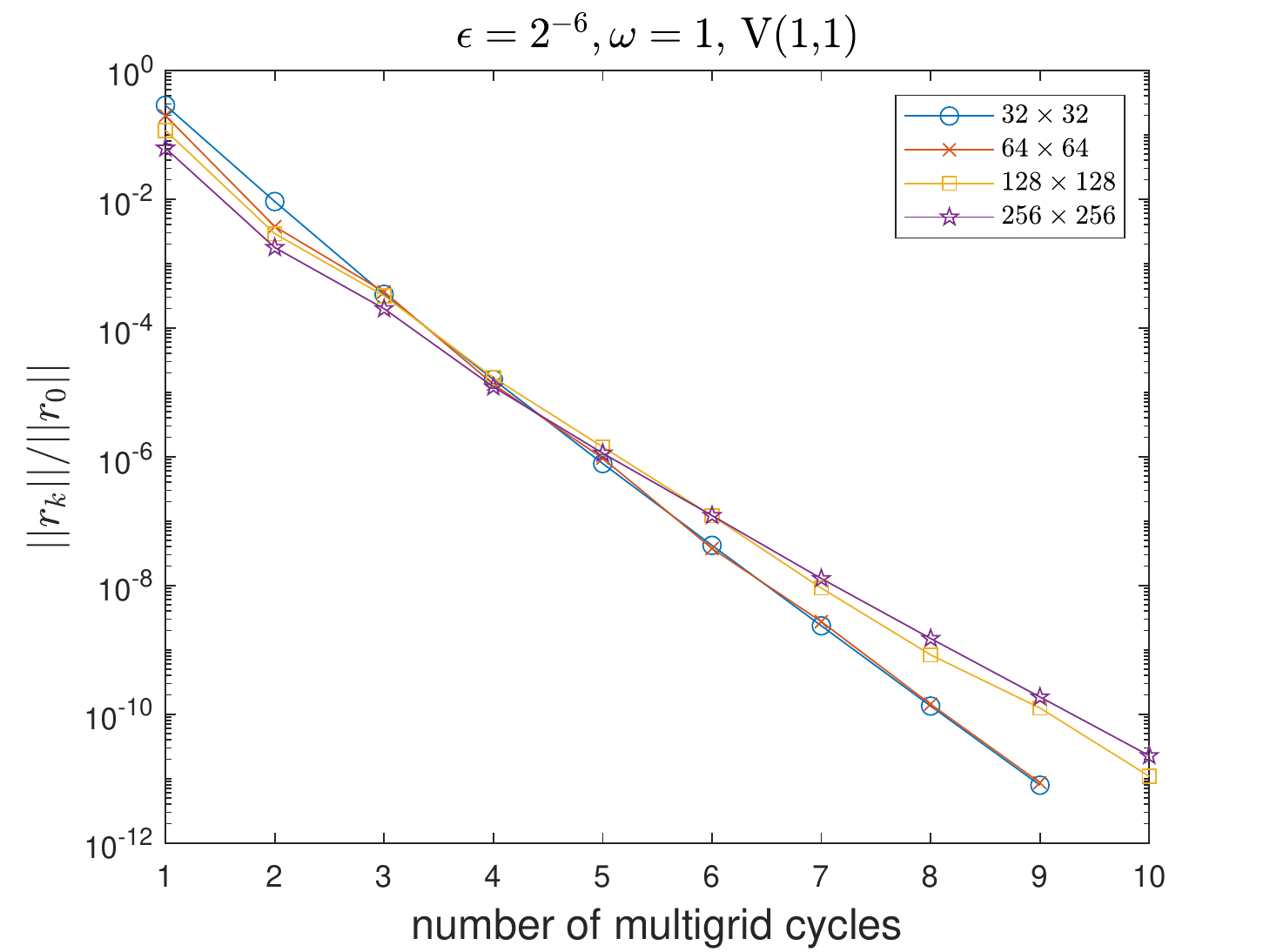}
\includegraphics[width=0.49\textwidth]{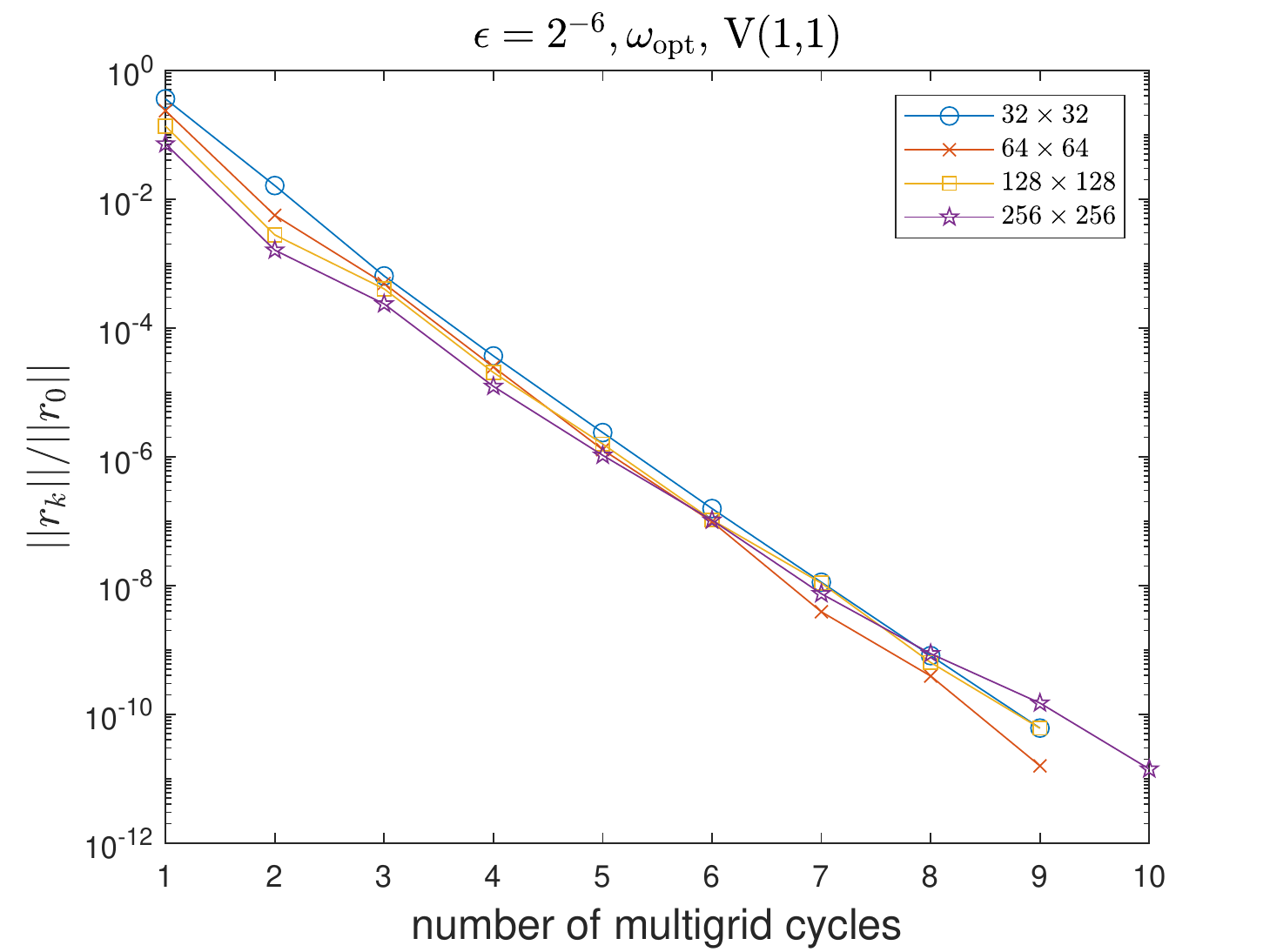}
 \caption{Convergence history: Number of iterations versus relative residual of V(1,1)-cycle with $\epsilon=2^{-6}$ and two Jacobi iterations for Schur complement system (left $\omega=1$ and right  optimal $\omega$).} \label{fig:V-vs-J2-eps6}
\end{figure}

\begin{figure}[h!]
\centering
\includegraphics[width=0.49\textwidth]{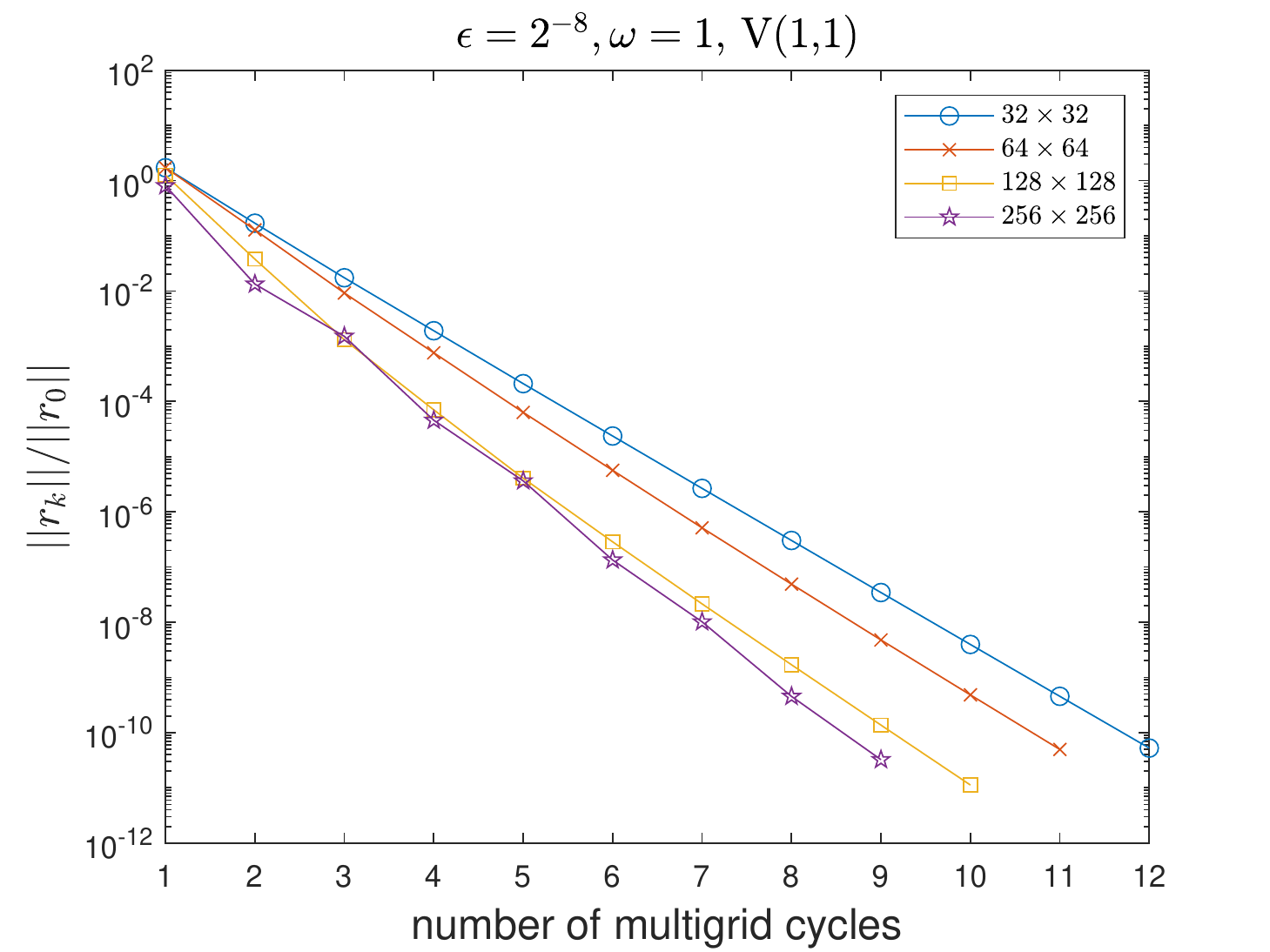}
\includegraphics[width=0.49\textwidth]{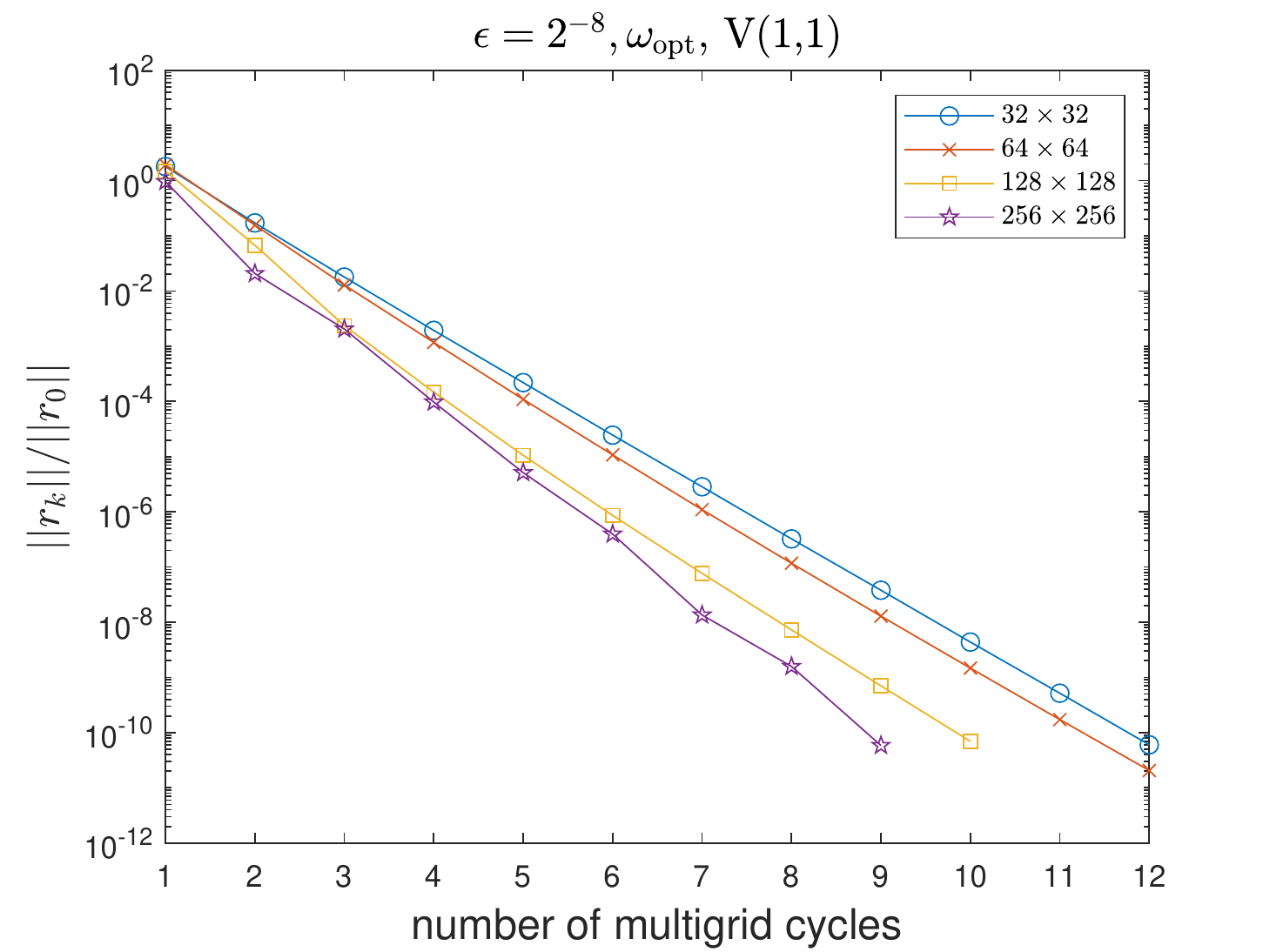}
 \caption{Convergence history: Number of iterations versus relative residual of V(1,1)-cycle with $\epsilon=2^{-8}$ and two Jacobi iterations for Schur complement system (left $\omega=1$ and right  optimal $\omega$).} \label{fig:V-vs-J2-eps8}
\end{figure}

{\bf Three iterations for Schur complement system:} We explore V(1,1)-cycle iterations with two choices of $\omega$ and a varying physical parameter $\epsilon$, using  three Jacobi iterations for Schur complement system.  We report the history of relative residual $\frac{||r_k||}{||r_0||}$  as a function of the  V(1,1)-cycle iteration counts for $n\times n$ meshgrid ($n=32,  64,128, 256$).  Figure \ref{fig:V-vs-eps0} shows the results for $\epsilon=1$. We see that using optimal $\omega$ takes 12 iterations of V(1,1)-cycle to achieve the stopping tolerance and it takes 13 iterations for $\omega=1$. We see that the convergence behavior is independent of meshsize $h$.  A similar performance is seen for  $\epsilon=2^{-2}, 2^{-4}, 2^{-6}, 2^{-8}$ in Figures \ref{fig:V-vs-eps2}, \ref{fig:V-vs-eps4},  \ref{fig:V-vs-eps6} and  \ref{fig:V-vs-eps8}.  Compared with two Jacobi iterations for solving the Schur complement system, three Jacobi iterations give a slightly better results for small  $\epsilon=2^{-6}, 2^{-8}$.  Again using optimal $\omega$  has   one iteration number  less than that  of $\omega=1$.  Thus, it is simple and reasonable to use $\omega=1$ in practice. Moreover, two Jacobi iterations are enough to achieve robustness V(1,1)-cycle multgrid with respect to meshgrid and physical parameter $\epsilon$.  

\begin{figure}[h!]
\centering
\includegraphics[width=0.49\textwidth]{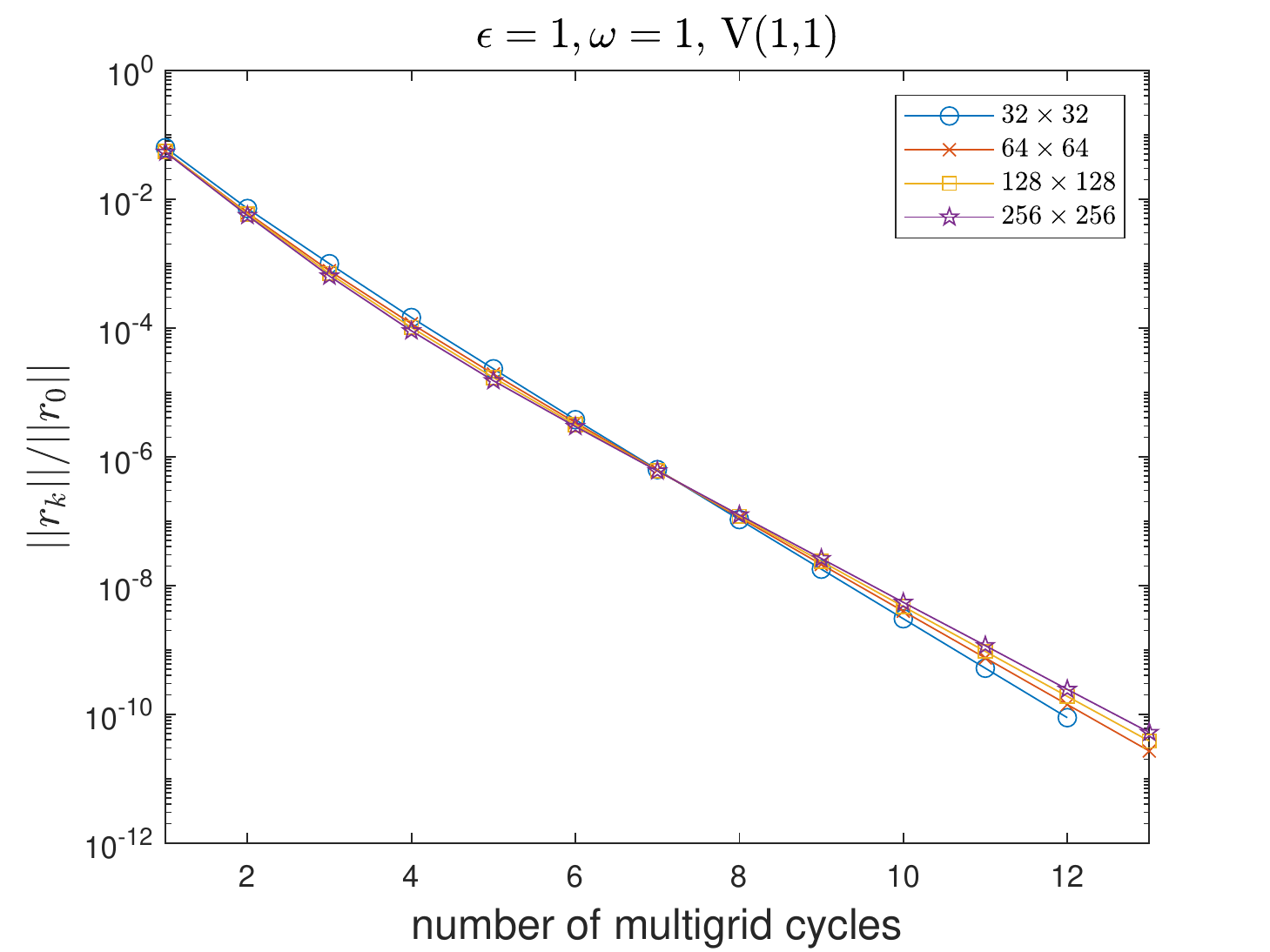}
\includegraphics[width=0.49\textwidth]{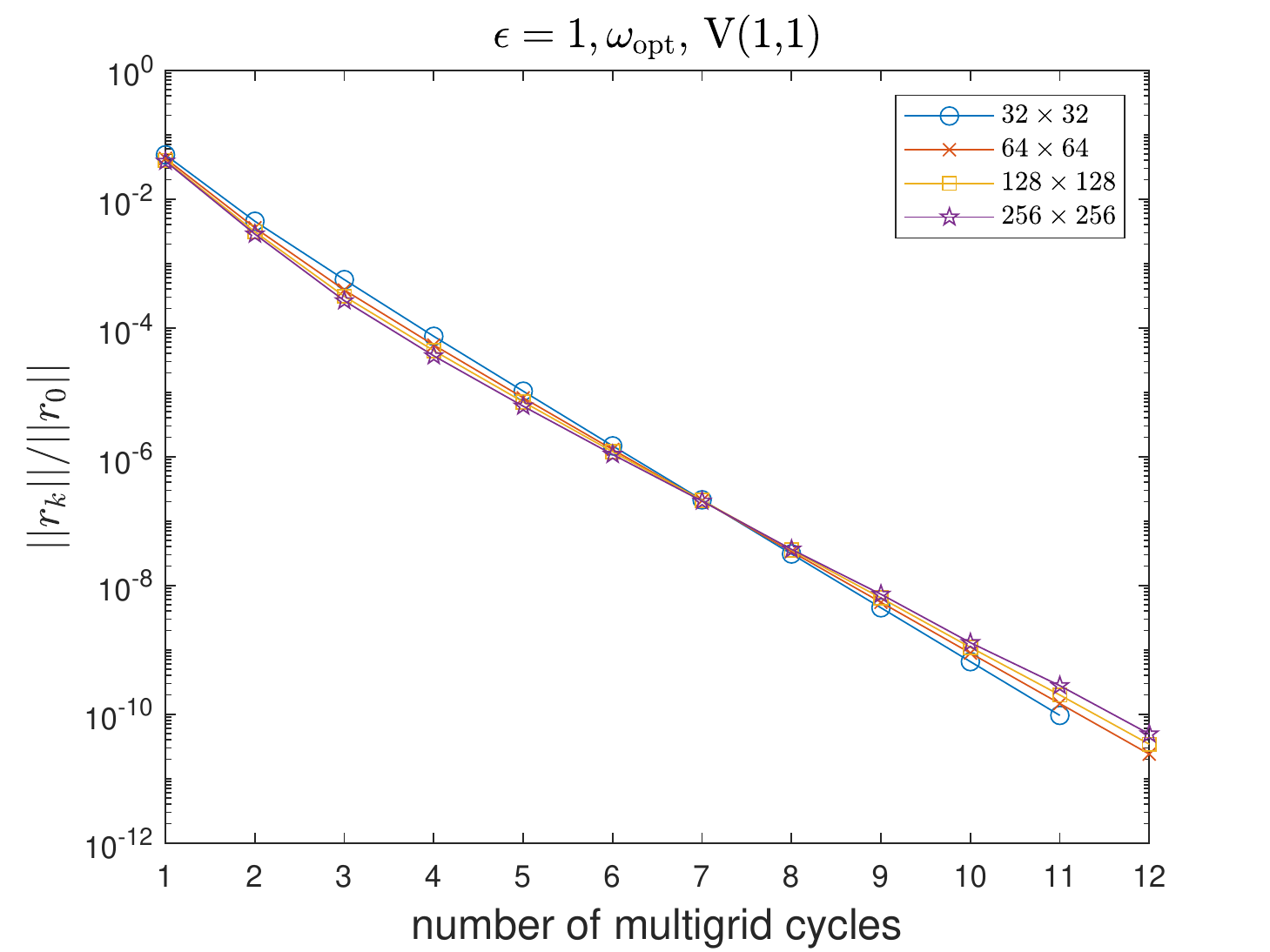}
 \caption{Convergence history: Number of iterations versus relative residual of V(1,1)-cycle with $\epsilon=1$ and three Jacobi iterations for Schur complement system (left $\omega=1$ and right  optimal $\omega$).} \label{fig:V-vs-eps0}
\end{figure}

\begin{figure}[h!]
\centering
\includegraphics[width=0.49\textwidth]{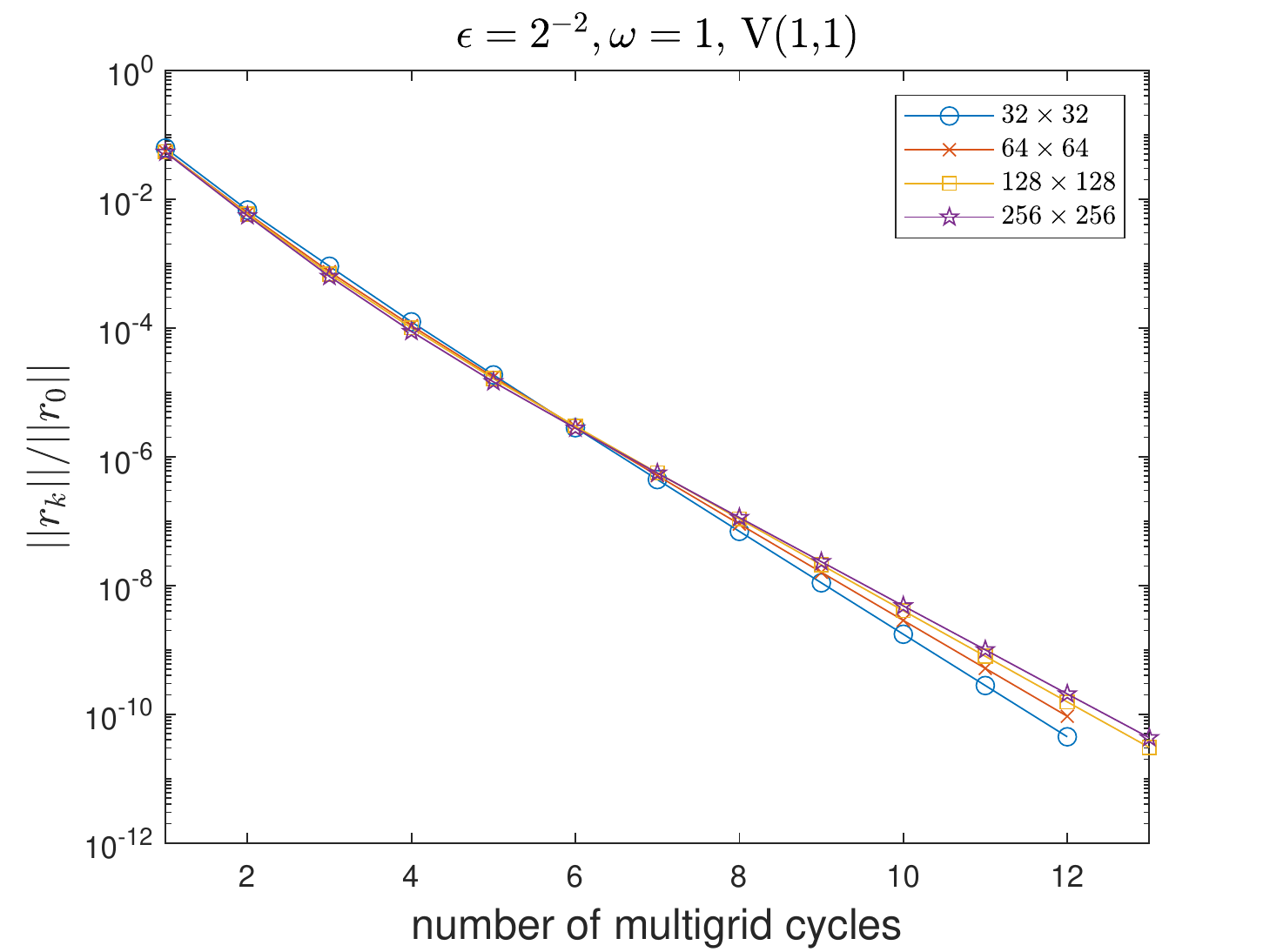}
\includegraphics[width=0.49\textwidth]{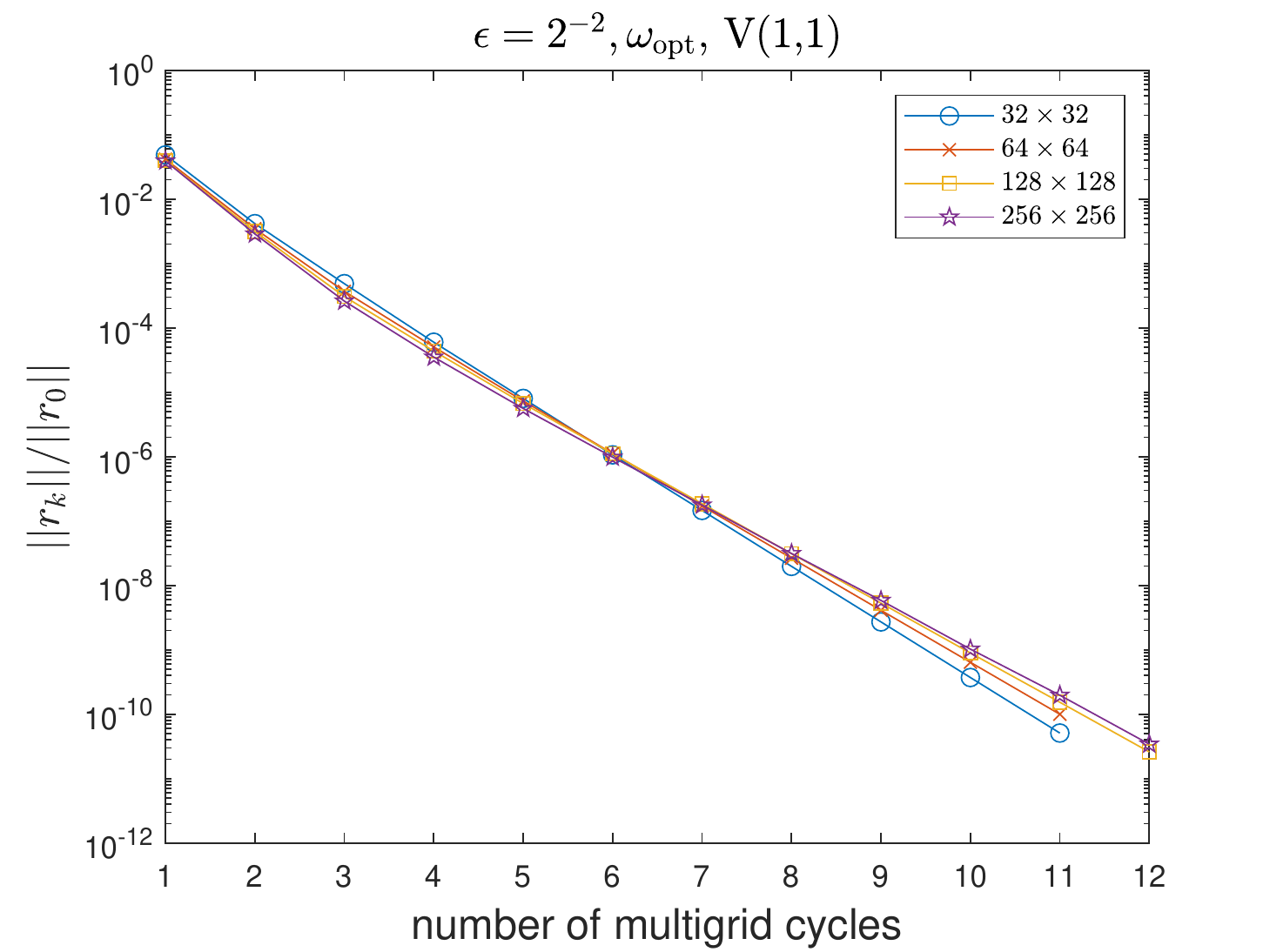}
 \caption{Convergence history: Number of iterations versus relative residual of V(1,1)-cycle with $\epsilon=2^{-2}$ and three Jacobi iterations for Schur complement system (left $\omega=1$ and right  optimal $\omega$).} \label{fig:V-vs-eps2}
\end{figure}

\begin{figure}[h!]
\centering
\includegraphics[width=0.49\textwidth]{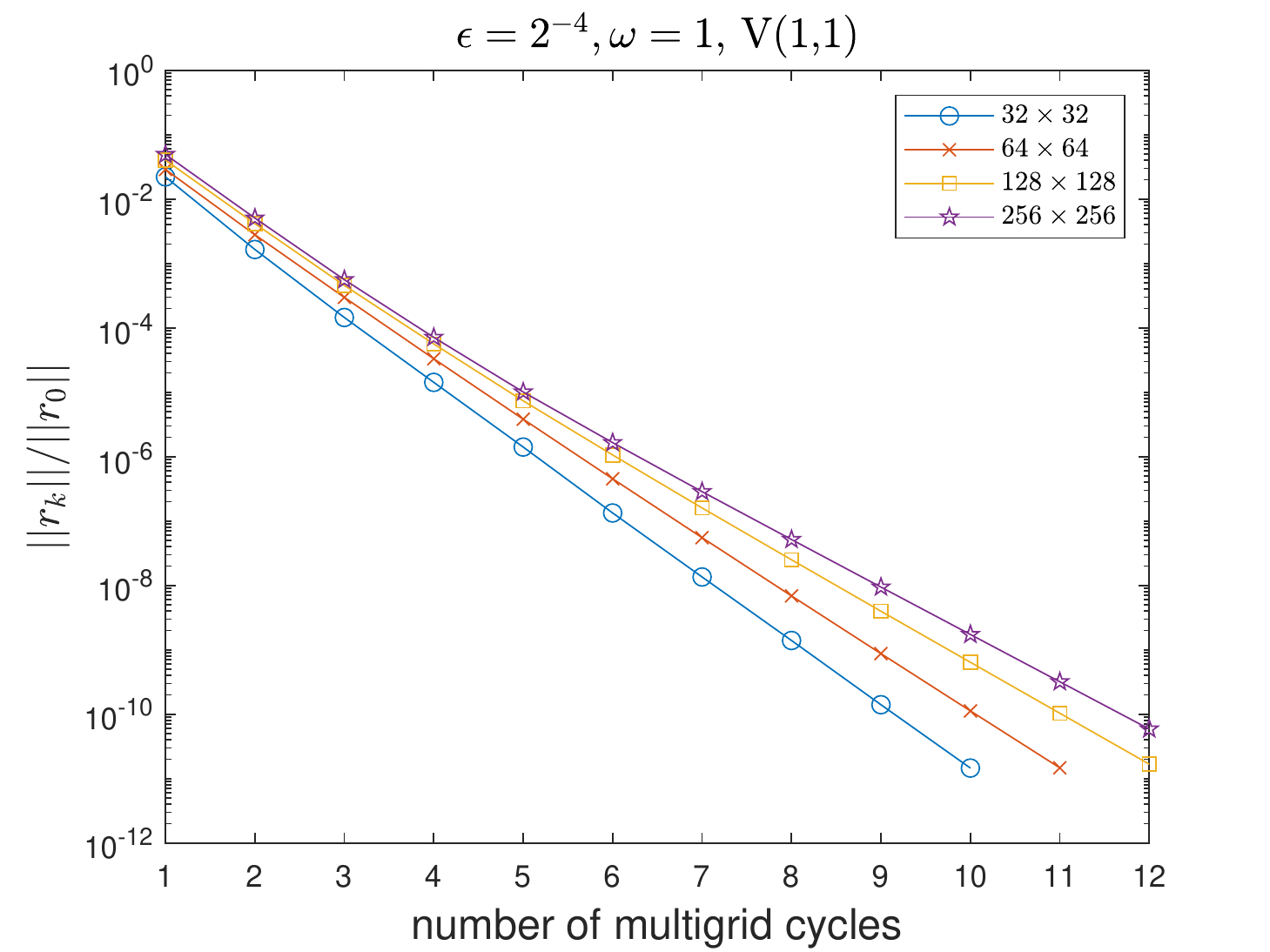}
\includegraphics[width=0.49\textwidth]{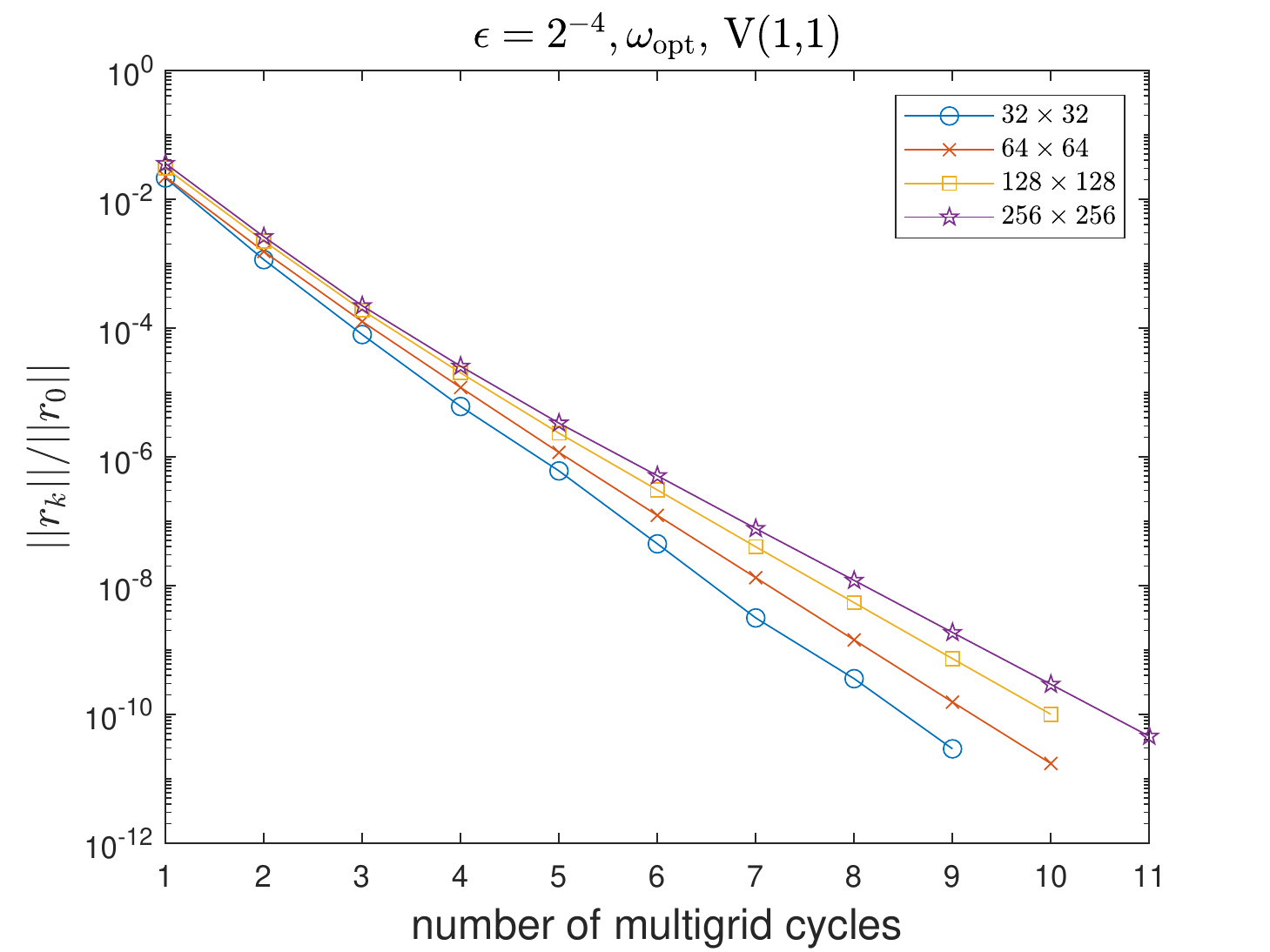}
 \caption{Convergence history: Number of iterations versus relative residual  of V(1,1)-cycle with $\epsilon=2^{-4}$ and three Jacobi iterations for Schur complement system (left $\omega=1$ and right  optimal $\omega$).} \label{fig:V-vs-eps4}
\end{figure}

\begin{figure}[h!]
\centering
\includegraphics[width=0.49\textwidth]{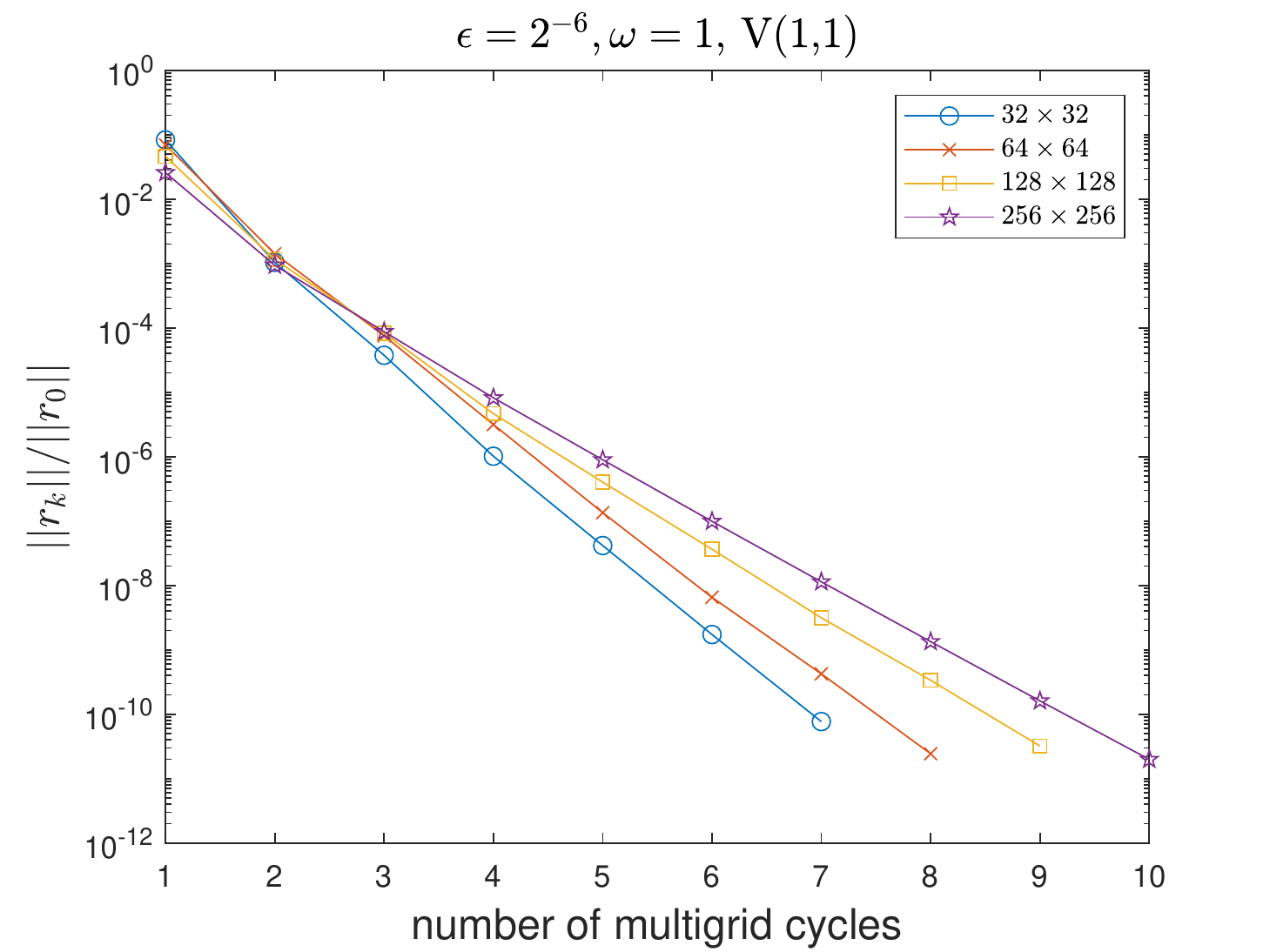}
\includegraphics[width=0.49\textwidth]{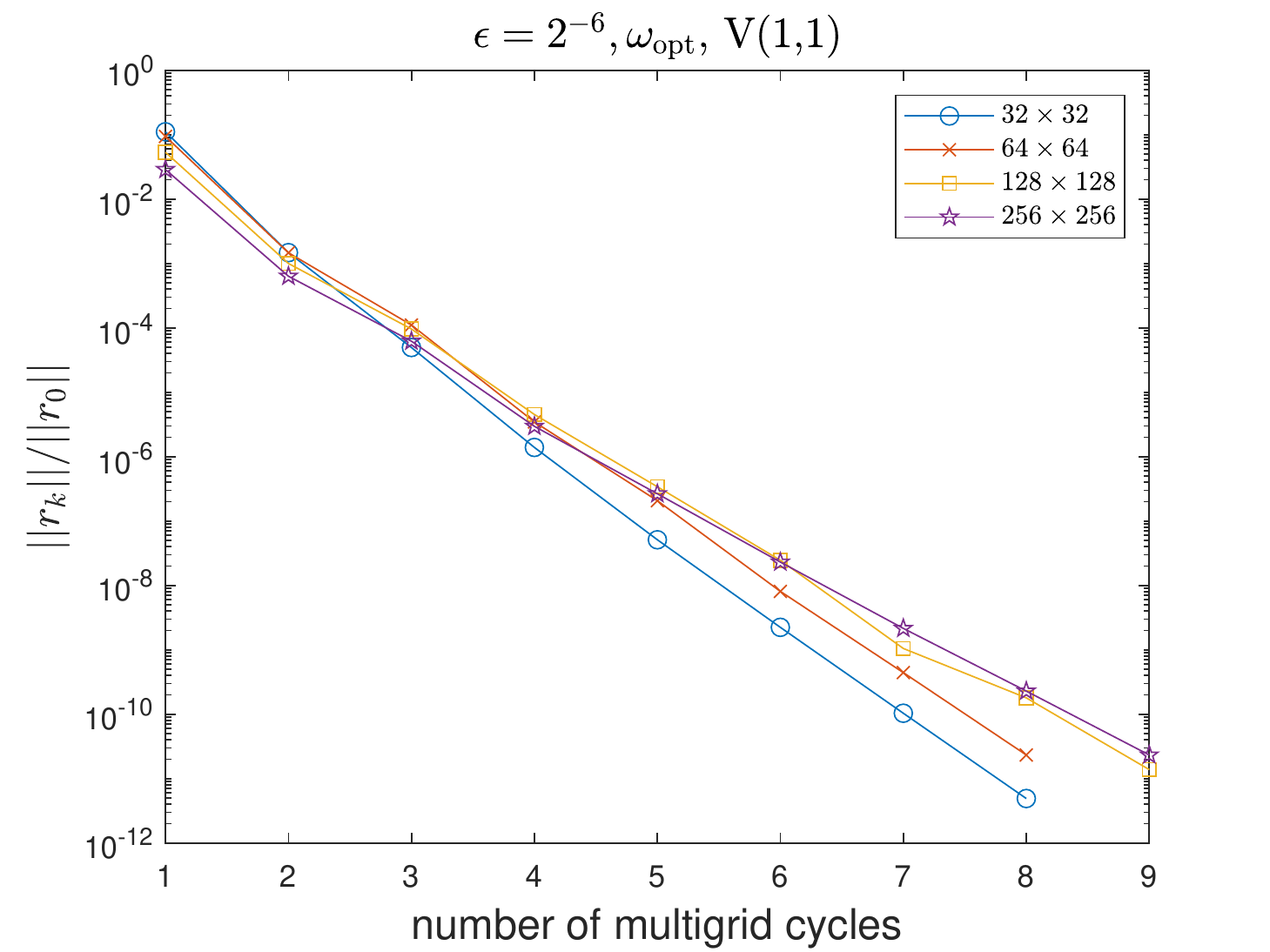}
 \caption{Convergence history: Number of iterations versus relative residual  of V(1,1)-cycle with $\epsilon=2^{-6}$ and three Jacobi iterations for Schur complement system (left $\omega=1$ and right  optimal $\omega$).} \label{fig:V-vs-eps6}
\end{figure}

\begin{figure}[h!]
\centering
\includegraphics[width=0.49\textwidth]{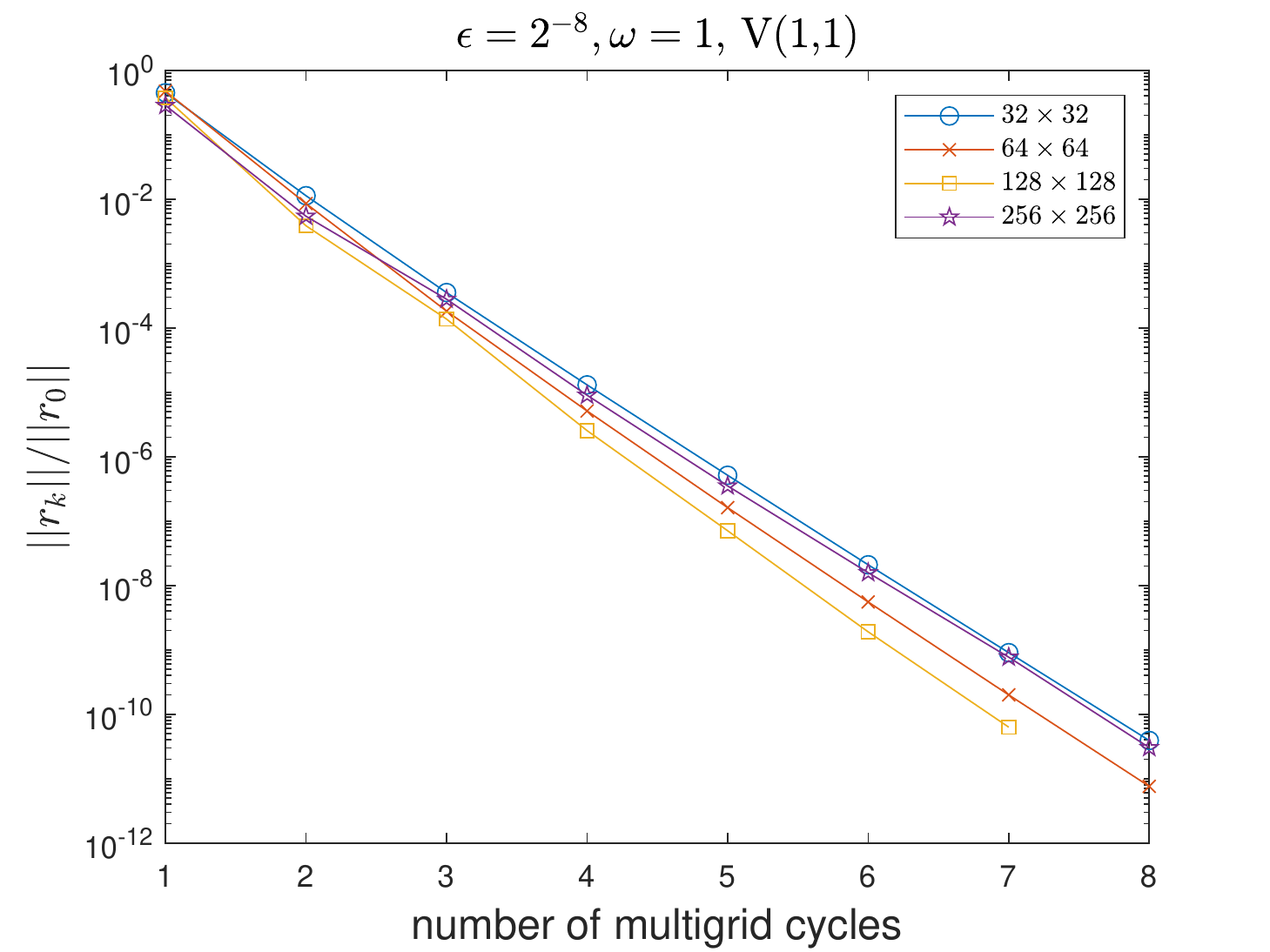}
\includegraphics[width=0.49\textwidth]{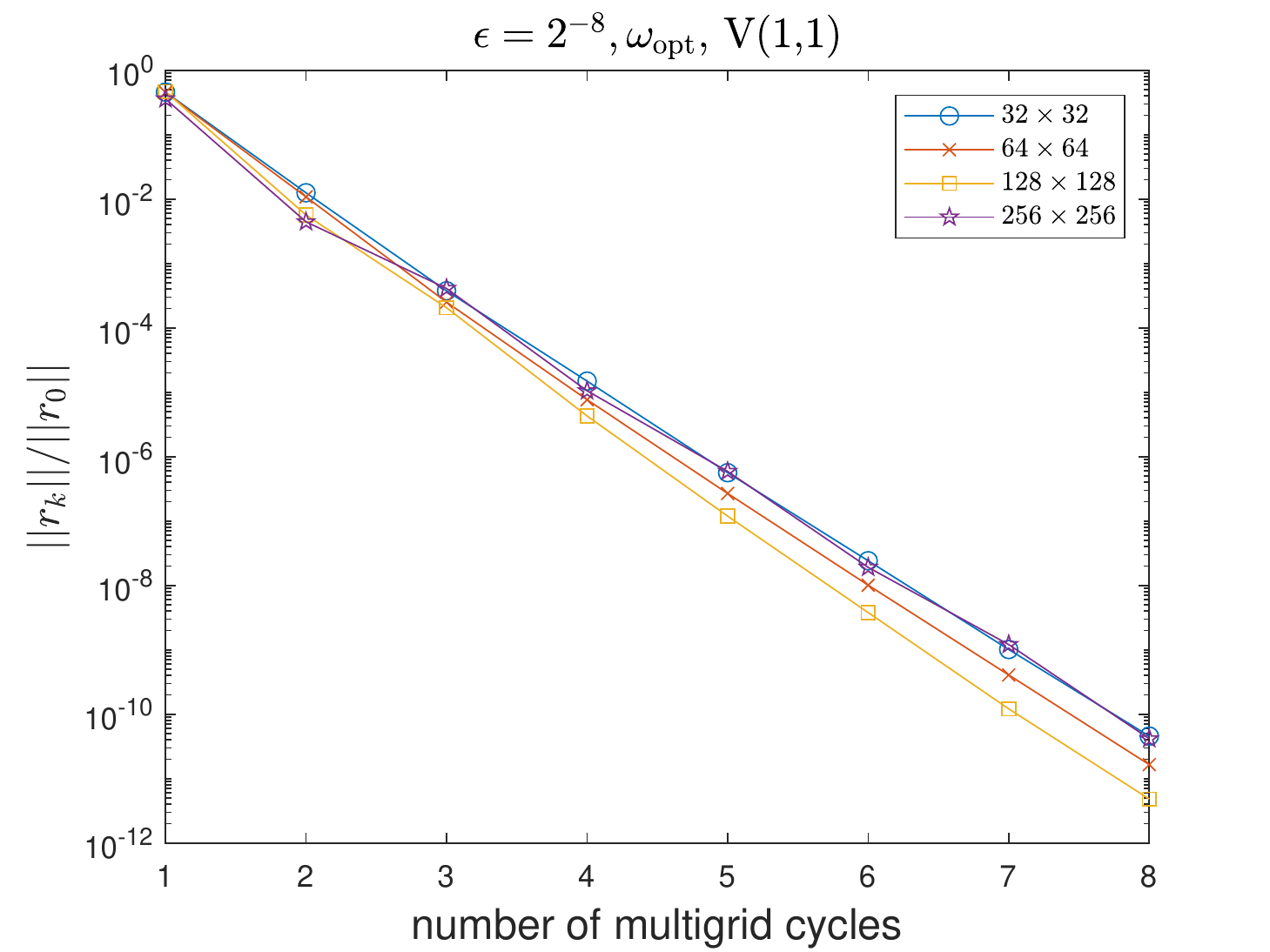}
 \caption{Convergence history: Number of iterations versus relative residual of V(1,1)-cycle with $\epsilon=2^{-8}$ and three Jacobi iterations for Schur complement system (left $\omega=1$ and right  optimal $\omega$).} \label{fig:V-vs-eps8}
\end{figure}

 \section{Conclusions}\label{sec:con}

 We propose a parameter-robust multigrid method for solving the discrete system of Stokes-Darcy problems, with the maker and cell scheme used for the discretization. The resulting linear system is a saddle-point system.  In contrast to existing Vanka smoothers, where the velocities and pressure unknowns in a grid cell are updated simultaneously,  we propose Vanka-based Braess-Sarazin relaxation scheme, where the Laplace-like term in the saddle-point system is solved by an additive Vanka algorithm.  This approach decouples the velocities and pressure unknowns. Moreover, only matrix-vector products are needed in our proposed multigrid method.   LFA is used to analyze the smoothing process and help choose the optimal parameter that minimizing LFA smoothing factor. From LFA, we derive the stencil of  additive Vanka for the Laplace-like operator, which can help form the global iteration matrix, avoiding solving many subproblems in the classical additive Vanka setting. 
 
 Our main contribution is that we derive the optimal algorithmic parameter and optimal LFA smoothing factor for Vanka-based Braess-Sarazin relaxation scheme, and show that this scheme is highly efficient with respect to physical parameter. Our theoretical results reveal that although the optimal damping parameter is related to physical parameter and meshsize, it is very close to one. We also present the theoretical LFA smoothing factor with damping parameter one. In Vanka-based Braess-Sarazin relaxation, we have to solve a Schur complement system. Direct solver is  often expensive.  We propose an inexact version of Vanka-based Braess-Sarazin relaxation, where we apply only two or three iterations of Jacobi to the Schur complement system to achieve the same performance as that of an exact solve.   We show that using damping parameter one can achieve almost the same performance as that of optimal result, and the results are close to exact version. Thus, using damping parameter one is recommended.  Our V-cycle multigrid   illustrates  high efficiency of our relaxation scheme and robustness to physical parameter.  
 
We comment that the proposed Vanka-based Braess-Sarazin multigrid method can be used
as a preconditioner for Krylov subspace methods. We have limited ourselves for  the MAC scheme on uniform grids. However,  it is possible to extend the Vanka-smoother to  non-uniform grids.  
 
\bibliographystyle{siam}
\bibliography{StokesDB_bib}


\end{document}